\documentclass{amsart}

\usepackage[latin1]{inputenc}
\usepackage{amsfonts}
\usepackage{amsmath}
\usepackage{amsthm}
\usepackage{amssymb}
\usepackage{latexsym}
\usepackage{enumerate}

\newtheorem{theorem}{Theorem}%[section]
\newtheorem{lemma}[theorem]{Lemma}
\newtheorem{proposition}[theorem]{Proposition}

\newtheorem{question}[theorem]{Question}

\newtheorem{definition}[theorem]{Definition}

\newtheoremstyle{note}% <name>
{3pt}% <Space above>
{3pt}% <Space below>
{}% <Body font>
{}% <Indent amount>
{\bf}% <Theorem head font>
{:}% <Punctuation after theorem head>
{.5em}% <Space after theorem headi>
{}% <Theorem head spec (can be left empty, meaning `normal')>

\theoremstyle{note}
\newtheorem{remark}[theorem]{Remark}

\numberwithin{equation}{section}

\begin{document}

\newcommand{\ad}{\mathsf{a.d.}}
\newcommand{\cc}{\mathfrak{c}}
\newcommand{\N}{\mathbb{N}}
\newcommand{\BB}{\mathbb{B}}
\newcommand{\C}{\mathbb{C}}
\newcommand{\Q}{\mathbb{Q}}
\newcommand{\R}{\mathbb{R}}
\newcommand{\st}{*}
\newcommand{\PP}{\mathbb{P}}
\newcommand{\rin}{\right\rangle}
\newcommand{\SSS}{\mathbb{S}}
\newcommand{\forces}{\Vdash}
\newcommand{\dom}{\text{dom}}
\newcommand{\osc}{\text{osc}}
\newcommand{\F}{\mathcal{F}}
\newcommand{\A}{\mathcal{A}}
\newcommand{\B}{\mathcal{B}}
\newcommand{\I}{\mathcal{I}}
\newcommand{\X}{\mathcal{X}}
\newcommand{\Y}{\mathcal{Y}}
\newcommand{\G}{\mathcal{G}}
\newcommand{\Z}{\mathcal{Z}}
\newcommand{\CC}{\mathcal{C}}
\newcommand{\T}{{\infty,2}}

\thanks{The research of the first named author was supported by a PAPIIT grant IA102222}
\thanks{{The research of the second named author was supported by a CONACyT grant A1-S-16164 and a PAPIIT grant IN104220..}}
\thanks{The third named author was partially supported by the NCN (National Science
Centre, Poland) research grant no.\ 2020/37/B/ST1/02613.}

\subjclass[2010]{}
\title[AD families and the geometry of spheres]{Almost disjoint families and the geometry of nonseparable spheres}

\author{Osvaldo Guzm\'{a}n}
\address{Centro de Ciencas Matem\'aticas\\
UNAM\\
A.P. 61-3, Xangari, Morelia, Michoac\'an\\
58089, M\'exico}
\email{oguzman@matmor.unam.mx}

\author{Michael Hru\v{s}\'{a}k}
\address{Centro de Ciencas Matem\'aticas\\
UNAM\\
A.P. 61-3, Xangari, Morelia, Michoac\'an\\
58089, M\'exico}
\email{michael@matmor.unam.mx}

\author{Piotr Koszmider}
\address{Institute of Mathematics, Polish Academy of Sciences,
ul. \'Sniadeckich 8,  00-656 Warszawa, Poland}
\email{\texttt{piotr.koszmider@impan.pl}}

\begin{abstract} 
We consider uncountable almost disjoint families of subsets of $\N$, the
Johnson-Lindenstrauss Banach spaces $(\X_\A, \|\ \|_\infty)$ induced by them, and
their natural equivalent renormings $(\X_\A, \|\ \|_\T)$.
We introduce a partial order $\PP_\A$ and characterize some geometric properties
of the spheres of $(\X_\A, \|\ \|_{\infty})$  and of $(\X_\A, \|\ \|_\T)$ in terms of combinatorial properties 
of $\PP_\A$.  This provides a  range of independence  and absolute
results concerning the above mentioned geometric properties via
combinatorial properties of   almost disjoint families. 
Exploiting  the extreme behavior of some known and some new
almost disjoint families we show the existence of Banach spaces where the unit spheres display 
surprising geometry:
\begin{enumerate}
\item There is a Banach space of density continuum whose unit sphere 
is the union of countably many sets of diameters strictly less than $1$.
\item It is consistent that for every $\rho>0$ there is a nonseparable Banach space, where 
for every $\delta>0$ there is $\varepsilon>0$ such that 
every uncountable $(1-\varepsilon)$-separated set  
of elements of the unit sphere 
contains two elements distant by less than $1$ and two elements distant at least by $2-\rho-\delta$.
%\item There is a nonseparable Banach space $\X$ and its separable subspace $\Y$ 
%with the following properties:  
%\begin{itemize}
%\item for every $\delta>0$ there is $\varepsilon>0$ such that whenever
% $\mathcal Z\subseteq \X$ is such that $\Z/\mathcal Y$ is $(1-\varepsilon)$-separated  in $\X/\Y$, then
%$\Z=\bigcup\Z_n$  where $\mathcal Z_n$ is   $(2-\delta)$-separated for each $n\in\N$.
%\item $\X/\Y$ does not admit an uncountable $(1+\varepsilon)$-separated set for any $\varepsilon>0$.
%\end{itemize}
\end{enumerate}
It should be noted that  for every $\varepsilon>0$  every nonseparable Banach space has a plenty of  
uncountable $(1-\varepsilon)$-separated  
sets by the Riesz Lemma.

We also obtain a consistent dichotomy for the spaces of the form $(\X_\A, \|\ \|_\T)$: The Open Coloring Axiom implies that
the unit sphere of every Banach space of the form $(\X_\A, \|\ \|_\T)$ 
either is the union of countably many sets of diameter strictly less than $1$ or it contains an
 uncountable $(2-\varepsilon)$-separated set for every $\varepsilon>0$.

\end{abstract}

\maketitle

\section{introduction}

A family $\A$ of subsets of $\N$ will be called an $\ad$-family if it is uncountable and consists of infinite sets
which are pairwise almost disjoint, i.e., $A\cap B$ is finite for any two distinct $A, B\in \A$.
Such $\ad$-families induce diverse topological and abstract analytic structures which  constitute
interesting (counter)examples ( e.g., \cite{akemann-doner, tristan, sailing, oursacks, michael-ad, supermrowka})
 in various  parts of  mathematics. 
In this paper we provide applications in the geometry  of the spheres of Banach spaces.

It was W. Johnson and N. Lindenstrauss who first
considered a natural Banach space induced by an $\ad$-family $\A$. We denote it as $\X_\A$.
It is the closure of the linear span in $\ell_\infty$ with the supremum norm (denoted $\|\ \|_\infty$) of
the set
$$c_0\cup\{1_A: A\in \A\},$$
where $1_A$ stands for the characteristic function of $A\subseteq \N$. It is easy to see that $\X_\A$ is linearly isometric
to the Banach space $C_0(\Psi_\A)$ of all real-valued continuous functions on $\Psi_\A$ vanishing at infinity, where
$\Psi_\A$ is the locally compact space induced by $\A$,  first considered by Alexandroff and Urysohn,
and later often called Mr\'owka-Isbell space (\cite{michael-psi}). 
In this paper we also consider $\X_\A$ with a different, but equivalent norm $\|\ \|_\T$ defined 
for $f\in \X_\A$ by
$$\|f\|_{\T}=\|f\|_\infty+ \sqrt{\sum_{n\in \N} {f(n)^2\over 2^{n+1}}}.$$
Analyzing 
the dependence of $(\X_\A, \|\ \|_\T)$ on the combinatorial properties of the
$\ad$-family $\A$  we obtain new examples of nonseparable Banach spaces where
the geometry of the unit sphere  features completely unknown until now character.

\begin{theorem}\label{thm-sigma} There is a Banach space $(\X, \|\ \|)$ of density continuum (of the form $(\X_\A, \|\ \|_\T)$ for
some $\ad$-family $\A$), where the
unit sphere is the union of countably many sets, each of diameter strictly less than $1$.
\end{theorem}
\begin{proof}
It is well known that there are $\ad$-families $\A$ of cardinality $\mathfrak c$
which are $\R$-embeddable (see definition \ref{def-Rembed}). By Proposition \ref{Rembed-centered}
the splitting partial order $\PP_\A$ (Definition \ref{def-P}) for $\A$ which is $\R$-embeddable
is $\sigma$-centered. 
Now Proposition \ref{sigma-T} yields the Theorem.
\end{proof}

This is a strengthening of the result of \cite{pk-kottman} where first nonseparable Banach spaces
with no uncountable $(1+)$-separated sets in the sphere (i.e., such that the distances between distinct
points are strictly bigger than $1$) were constructed.  Clearly the above space has this property as well.
As in \cite{pk-kottman} one can prove that
such spaces  do not admit uncountable Auerbach systems or uncountable equilateral sets 
(sets where distances between any two points are the same)  or the unit ball cannot be packed
with uncountably many balls of diameter $1/3$. 
Also in \cite{pk-kottman} it was shown that there are $\ad$-families (known as Luzin families)
such that the unit sphere of $(\X_\A, \|\ \|_\T)$ does contain an uncountable $(1+)$-separated set.
In fact, here we obtain the following consistent dichotomy (here $(2-\varepsilon)$-separated means that
any two distinct points are distant at least by $2-\varepsilon$):

\begin{theorem}\label{thm-oca} Assume the Open Coloring Axiom {\sf OCA}.
Suppose that $\A$ is an $\ad$-family. Then the unit sphere of 
the Banach space $(\X_\A, \|\ \|_\T)$  either is the union of countably many sets, 
each of diameter strictly less than $1$ or it contains an uncountable $(2-\varepsilon)$-separated set
for every $\varepsilon>0$.
\end{theorem}
\begin{proof} Use proposition \ref{oca-T}.
\end{proof}

It follows from Proposition \ref{antiramsey-T} that the above dichotomy is not provable in {\sf ZFC}
for any $\varepsilon>0$.
Most of the results of this paper should be seen in the context of the Riesz lemma of 1916 (\cite{riesz}) which
says that given a Banach space $\X$,  and its closed proper subspace $\Y\subseteq \X$ and $\varepsilon>0$
there is $x$ in the unit sphere of $\X$ such that $\|x-y\|>1-\varepsilon$ for every $y\in\Y$. This allows to
construct $(1-\varepsilon)$-separated sets of cardinality $\kappa$ in any Banach space of density
equal to $\kappa$, where $\kappa$ is an infinite cardinal.  It was Kottman who showed that the unit sphere of any
separable Banach space contains an infinite $(1+)$-separated set (\cite{kottman}) and Elton and Odell
who showed that it always contains an  infinite $(1+\varepsilon)$-separated set  for some $\varepsilon>0$ (\cite{elton-odell}).
The search for a nonseparable version of Kottman's result produced many deep postitive partial results which
culminated in \cite{hajek-tams} but ended in the negative result of \cite{pk-kottman}. Theorem \ref{thm-sigma}
provides even stronger failure of any nonseparable version of Kottman's result. Theorem \ref{thm-oca} shows
that we can uniformize,  at least consistently, the behavior of $(1-\varepsilon)$-separated sets of the spheres of the spaces 
$(\X_\A, \|\ \|_\T)$ for some $\ad$-family  $\A$ obtaining uniform behavior on an uncountable set.
We consistently obtain Banach spaces
 where the behavior of the norm on the unit sphere  is far from uniform
even on each uncountable $(1-\varepsilon)$-separated set:

\begin{theorem}\label{thm-anti} It is consistent (e.g., follows from {\sf CH}) that 
for every $\rho>0$ there exist a nonseparable Banach space $(\X, \|\ \|)$ (a subspace of
a space of the form $(\X_\A, \|\ \|_\T)$ for some $\ad$-family  $\A$)
 such that for every $\delta>0$ there is $\varepsilon>0$ such that whenever and $\Y\subseteq \X$ is
a $(1-\varepsilon)$-separated subset of the unit sphere of $\X$,  then
\begin{itemize}
\item there are distinct $x, y\in \Y$ such that $\|x-y\|<1$, and
\item there are $x, y\in \Y$ such that $\|x-y\|>2-\delta-\rho$.
\end{itemize}
\end{theorem}
\begin{proof} By either of the Propositions \ref{antiramsey-gen}, \ref{antiramsey-ch}, \ref{antiramsey-cohen}
it is consistent that antiramsey $\ad$-families exist.
Now Proposition \ref{antiramsey-Tphi} yields the theorem.
\end{proof}

Note that analogous phenomena cannot occur for infinite instead of uncountable 
$(1-\varepsilon)$-subsets of the unit spheres
since by Theorem 1 of \cite{mer-pams} every  Banach space admits an equivalent renorming arbitrarily close to the original one
which does admit $1$-equilateral subset of the sphere (i.e., such a set where distances
between any two distinct points are equal to $1$). Consequently every infinite dimensional
unit sphere in a Banach spaces admits an infinite  subset $\mathcal Y$ such that $1-\varepsilon\leq\|y-y'\|\leq 1+\varepsilon$
for any distinct $y, y'\in \Y$ (while even nonseparable Banach spaces may not admit infinite equilateral sets - see
\cite{pk-wark}).

By choosing yet another $\ad$-family
we can obtain the geometry of the sphere completely opposite to the one from Theorem \ref{thm-sigma}:

\begin{theorem}\label{thm-antisigma} There exist a 
nonseparable Banach space $(\X, \|\ \|)$ (a subspace of
a space of the form $(\X_\A, \|\ \|_\infty)$ for some $\ad$-family  $\A$) and its separable subspace $\Y$
 such that 
 \begin{itemize}
 \item for every $\delta>0$ there is $\varepsilon>0$ such that whenever $\Z\subseteq \X$ is 
a subset of the unit sphere of $\X$ such that $\{[z]_\Y: z\in\Z\}$ is a $(1-\varepsilon)$-separated subset of $\X/\Y$, then
$\Z$ is the union of countably many sets which are $(2-\delta)$-separated.
\item $\X/\Y$ does not admit an uncountable $(1+\varepsilon)$-separated set for any $\varepsilon>0$.
\end{itemize}
\end{theorem}
\begin{proof} It is well known that Luzin families exist (\cite{luzin, michael-ad}).
They are $L$-families by Proposition \ref{anti-centered}. Now 
Proposition \ref{Y-luzin} yields the first part of the theorem. 
For the second part note that $\Y$ obtained from  Proposition \ref{Y-luzin} is $c_0$
and $\X_\A/c_0$ is isometric to $c_0(\omega_1)$ where it is well known that there
are no uncountable $(1+\varepsilon)$-separated sets (\cite{elton-odell}) for any $\varepsilon>0$.
\end{proof}

As in \cite{pk-kottman} our main method to obtain the above Banach spaces is a combinatorial analysis
of underlying $\ad$-families which is greatly refined here.  It turned out that the right tool for this analysis
is the following partial order associated with an $\ad$-family:

\begin{definition}\label{def-P} Suppose that $\A$ is  an $\ad$-family. The splitting partial order $\PP_\A$ consists of
all pairs $p=(A_p, B_p)$ such that $A_p, B_p\subseteq\N$, $A_p\cap B_p=\emptyset$ and there are
finite subsets $a_p, b_p$ of $\A$ such that $A_p=^*\bigcup a_p$ and $B_p=^*\bigcup b_p$.

We say that $p\leq q$ if $A_p\supseteq A_q$ and $B_p\supseteq B_q$.

We call  $p, q\in \PP_\A$ essentially distinct if
$a_p\not=a_q$ and $b_p\not=b_q$.
\end{definition}

Although we do not force with the order $\PP_\A$, we use forcing terminology, i.e., $p, q\in \PP_\A$
are compatible if there is $r\in \PP_\A$ such that $r\leq p, q$ and otherwise they are incompatible.
Elements of $\PP_\A$ will be called conditions.
$\PP_\A$ satisfies the c.c.c. if it does not admit an uncountable pairwise incompatible subset.
A subset $\PP\subseteq\PP_\A$ is centered if for any $p_1, \dots p_k\in \PP$ there is $q\in\PP_\A$
such that $q\leq p_1, \dots p_k$.
$\PP_\A$ has precaliber $\kappa$, for $\kappa$ a cardinal, if every subset of $\PP_\A$ of cardinality $\kappa$
contains a further subset of cardinality $\kappa$ which is centered. $\PP_\A$ is said to
be $\sigma$-centered if $\PP=\bigcup_{n\in \N}\PP_n$ and each $\PP_n$ is centered.  The core
of our arguments leading to Theorems \ref{thm-sigma} and \ref{thm-oca} is based
on the following:

\begin{theorem}\label{thm-equi} Suppose that $\A$ is an $\ad$-family.
\begin{enumerate}
\item The following are equivalent:
\begin{enumerate}
\item $\PP_\A$ satisfies the c.c.c.
\item The unit sphere of $(\X_\A, \|\ \|_\infty)$ does not admit
an uncountable $(1+\varepsilon)$-separated set for some (equivalently for $\varepsilon=1$) $\varepsilon>0$.
\item The unit sphere of $(\X_\A, \|\ \|_\T)$ does not admit
an uncountable $1$-separated set (equivalently $(2-\varepsilon)$-separated for some $\varepsilon>0$).
\end{enumerate}
\item The following are equivalent for a cardinal $\kappa$ of uncountable cofinality:
\begin{enumerate}
\item $\PP_\A$  has precaliber $\kappa$.
\item For every $\varepsilon>0$ and every subset $\Y$ of the unit sphere of $(\X_\A, \|\ \|_\infty)$ 
of cardinality $\kappa$ there is $\Z\subseteq \Y$ of cardinality $\kappa$  which has
diameter less than $1+\varepsilon$.

\item For every $\varepsilon>0$ and every subset $\Y$ of the unit sphere of $(\X_\A, \|\ \|_\T)$ 
of cardinality $\kappa$ there is $\Z\subseteq \Y$ of cardinality $\kappa$  which has
diameter less than $1$.
\end{enumerate}
\item The following are equivalent
\begin{enumerate}
\item $\PP_\A$  is $\sigma$-centered.
\item For every $\varepsilon>0$ the unit sphere of $(\X_\A, \|\ \|_\infty)$ is the union of countably many sets
of diameters less than $1+\varepsilon$.
\item The unit sphere of $(\X_\A, \|\ \|_\T)$ is the union of countably many sets
of diameters less than $1$.
\end{enumerate}
\end{enumerate}
\end{theorem}
\begin{proof}
For (1) use Propositions \ref{sphere-ccc},  \ref{ccc-T} and \ref{sphereT-PA}.
For (2) use Propositions \ref{sphere-precaliber},  \ref{precaliber-T} and \ref{sphereT-PA}.
For (3) use Propositions \ref{sphere-sigma},  \ref{sigma-T} and \ref{sphereT-PA}.
\end{proof}
To capture the combinatorics of $\PP_\A$ behind Theorems \ref{thm-anti} and \ref{thm-antisigma} we introduce
the following:

\begin{definition}\label{def-antiramsey} An $\ad$-family $\A$
is called antiramsey if whenever an uncountable $\PP\subseteq \PP_\A$ consist of essentially distinct conditions, then there are
distinct $p, q\in \PP$ which are compatible and there $p, q\in \PP$ which are incompatible.
\end{definition}

\begin{definition}\label{def-LL}
An $\ad$-family $\A$
is called $L$-family if whenever $\PP\subseteq \PP_\A$ consist of essentially distinct conditions, then 
$\PP=\bigcup_{n\in \N}\PP_n$
and each $\PP_n$ consists of pairwise incompatible conditions. 
\end{definition}

To be able to take advantage of the links between the spaces
$(\X_\A, \|\ \|_\infty)$ and $(\X_\A, \|\ \|_\T)$ and the splitting partial order $\PP_\A$
we need to analyze the structure of the order $\PP_\A$ for various $\ad$-families $\A$.
This is the most important ingredient of this paper  allowing to prove the first four theorems.
Besided antiramsey $\ad$-families and $L$-families we  consider  in the context
of the above properties of the partial order $\PP_\A$ known types of almost disjoint families
such Luzin families, $\R$-embedddable families and families  admitting an $n$-Luzin gap and obtain the following:

\begin{theorem}\label{thm-ad} Suppose that $\A$ is an $\ad$-family.
\begin{enumerate}
\item
\begin{enumerate}
\item $\A$ does not contain an $n$-Luzin gap for any $n\in \N$ implies that $\PP_\A$ satisfies the c.c.c., which implies that
$\A$ does not contain a $2$-Luzin gap.
\item None of the above implications can be reversed in {\sf ZFC}.
\item The  conditions of (a) are equivalent under {\sf OCA}.
\end{enumerate}
\item
\begin{enumerate}
\item   If $\A$ is $\R$-embeddable, then $\PP_\A$ is $\sigma$-centered. So such $\ad$-families exist in {\sf ZFC}.
The implication cannot be reversed.
\item If $\A$ is  a maximal $\ad$-family, then $\PP_\A$ is not $\sigma$-centered.

\end{enumerate}
\item 
\begin{enumerate}
\item {\sf (MA)} Suppose that $\A$ is an $\ad$-family of cardinality less than $\mathfrak c$. Then 
either $\PP_\A$ does not satisfy c.c.c. or  $\PP_\A$ is $\sigma$-centered.
\item {\sf (OCA)} Suppose that $\A$ is an $\ad$-family. Then 
either $\PP_\A$ does not satisfy c.c.c. or  $\PP_\A$ is $\sigma$-centered.
\end{enumerate}

 \item If $\A$ is a Luzin family, then  it is an $L$-family. In particular $L$-families exist.
 \item It is consistent (and independent) that antiramsey $\ad$-families exist, for example under {\sf CH} and
 in many models of the negation of {\sf CH}. They may be maximal.
\end{enumerate}
\end{theorem}
\begin{proof}
1 (a): Propositions \ref{2L-notccc}, \ref{notccc-nL};1(b): Propositions  \ref{oca-Lccc}, \ref{oca-L};
1 (c): Propositions \ref{ccc3L}, \ref{notcccnot2L}; 2 (a) Propositions \ref{Rembed-centered}, \ref{notRembed-centered};
2 (b): \ref{mad-notcentered}; 3  (a): Proposition \ref{ma}; 3 (b) Proposition \ref{oca};  4: Proposition \ref{anti-centered};
5: Propositions \ref{antiramsey-gen}, \ref{antiramsey-ch}, \ref{antiramsey-cohen}.
\end{proof}

Theorems \ref{thm-equi} an \ref{thm-ad} joined together yield a myriad of corollaries. For example
under {\sf OCA} the unit sphere of the space $(\X_\A, \|\ \|_\T)$ admits a $1$-separated set if and only
if $\A$ admits a $2$-Luzin gap, or for no maximal $\ad$-family $\A$
the unit sphere of the space $(\X_\A, \|\ \|_\T)$ is the union of countably many sets
of diameters strictly less than $1$.

The structure of the paper is the following: Section 2 contains some preliminaries.
In Section 3 we investigate combinatorial properties of the forcing $\PP_\A$
in the context of  combinatorial properties of almost disjoint families $\A$,
in particular we obtain all elements needed for Theorem \ref{thm-ad}.
Section 4 is devoted to a graph which can code all the information of the
compatibility graph of $\PP_\A$ but the vertices of it are pairs of finite subsets
of the $\ad$-family $\A$.  This section is not used in other sections but provides an
alternative expression of the results of the other sections.
In Section 5 we link the geometric properties of the sphere of spaces $(\X_\A, \|\ \|_\infty)$
with the combinatorial properties of the splitting partial order $\PP_\A$. 
Equivalences like in Theorem \ref{thm-equi} are possible for antiramsey and $L$-families and the spaces
$(\X_\A, \|\ \|_\infty)$ which are
presented in Propositions \ref{C_0-antiramsey} and \ref{C_0-luzin}  and for spaces $(\X, \|\ \|_\T)$
but are less natural as one needs to
include additional technical conditions since the spheres of any infinite dimensional
Banach spaces contain uncountable sets of small diameters.  To avoid these technical conditions
in Section 6 we consider certain subspaces $\X_{\A,\phi}$ of the spaces $\X_\A$
for $\A$ being antiramsey and being an $L$-family and obtain Theorem \ref{thm-antisigma} for $\A$ being an $L$-family.
In Section 7 we investigate the spaces $\X_\A$ with the norm $\|\ \|_\T$,
in particular we provide the space needed for Theorem \ref{thm-sigma}.
In Section 8 we produce the space of Theorem \ref{thm-anti}.

\section{Preliminaries}
\subsection{Notation and terminology}

The $\ad$-families are defined in the introduction.  Antiramsey families and $L$-families are defined in
Defintions \ref{def-antiramsey} and \ref{def-LL}.

For $A, B\subseteq \N$ we write $A\subseteq^* B$ if $A\setminus B$ is finite; $A=^*B$ if $A\subseteq^*B$ and $B\subseteq^*A$.
By $1_A$ we mean the characteristic function of $A\subseteq\N$.
$[X]^{<\omega}$ denotes the family of all finite subsets of a set $X$. $f|X$ stands for the restriction of a
function $f$ to $X$.
By $2^{<\omega}$ we mean all finite $0$-$1$-sequences. 
 If $t\in 2^{<\omega}$, then $t^\frown i\in 2^{<\omega}$
for $i=0,1$ is the extension of $t$ by $i$.  The symbol $\mathfrak c$ stands for the cardinality of the continuum.
Martin's axiom is denoted {\sf MA}, the continuum hypothesis is denoted {\sf CH} and
the Zermelo-Fraenkel  theory is denoted by {\sf ZFC}.  For more information on {\sf MA}, {\sf CH}, {\sf ZFC}
see \cite{jech}.

The Open Coloring Axiom ({\sf OCA}) is the following statement: Given a metric separable $X$ and
a partition $K_0\cup K_1$ such that $\{(x, y): \{x, y\}\in K_0\}$ is open either there
is an uncountable $0$-homogeneous set  $Y\subseteq X$ or $X$ is the union of countably many $1$-homogeneous sets.
For more information on {\sf OCA} see \cite{jech}.

All Banach spaces considered in this paper are infinite dimensional and over the reals.
 $S_\X$
denotes the unit sphere of a Banach space $\X$. Sometimes we consider two norms
$\|\ \|_1$ and $\|\ \|_2$ on the same Banach spaces $\X$, then the corresponding spheres are
denoted by $S_{\X, \|\ \|_1}$ and $S_{\X, \|\ \|_2}$ respectively. $\|\ \|_\infty$
stands for the supremum norm while $\|\ \|_2$ is the standard norm in the Hilbert space $\ell_2$.
A subset $\Y$ of a Banach space $(\X, \|\ \|_\X)$ is called $\delta$-separated  ($(\delta+)$-separated, 
$\delta$-equilateral) for $\delta>0$
if $\|x-x'\|_\X\geq\delta$ ($\|x-x'\|_\X>\delta$, $\|x-x'\|_\X=\delta$) for any two distinct $x, x'\in \Y$. 
A subset $\Y$ of a Banach space $(\X, \|\ \|_\X)$ is called equilateral 
if it is $\delta$-equilateral for some $\delta>0$. Note that a Banach space $\X$
admits an uncountable (infinite) equilateral set if and only if it $S_\X$ admits an uncountable (infinite) 
$1$-equilateral set (scale the set and translate one of its elements to $0$).

An antichain in $\PP_\A$ is a set of pairwise incompatible conditions.
Some other forcing terminology is recalled below Definition \ref{def-P}.

All the undefined terminology should be standard and can be find in the following books:
\cite{engelking} for topology, \cite{fabianetal} for Banach spaces and \cite{jech} for set theory.

\subsection{The partial order $\PP_\A$}

Essentially distinct conditions are defined together with the order $\PP_\A$ in Definition \ref{def-P}.

\begin{lemma}\label{thinning} Suppose that $\A=\{A_\alpha: \alpha<\lambda\}$ is an $\ad$-family, $\lambda$ is a cardinal,  $\kappa$
 is a regular uncountable cardinal
 and $\{p_\xi: \xi<\kappa\}\subseteq\PP_\A$. Then there is $\Gamma\subseteq \kappa$ of cardinality $\kappa$
 and $\{p_\xi': \xi\in \Gamma\}\subseteq\PP_\A$ such that for every $\xi, \eta\in \Gamma$ we have
 that $p_\xi$ and $p_\eta$ are compatible in $\PP_\A$ if and only if $p_\xi'$ and $p_\eta'$ are compatible in $\PP_\A$
 and moreover  there are $k, l, m\in \N$ and $a_\xi'\in [\A]^k$, $b_\xi'\in [\A]^{l}$ for $\xi\in \Gamma$ and disjoint $E, F\subseteq m$
such that for all distinct $\xi, \eta\in \Gamma$ the following hold:
\begin{itemize}
\item $A_{p_\xi'}=(\bigcup a_\xi'\setminus m)\cup E$,
\item $B_{p_\xi'}=(\bigcup b_\xi'\setminus m)\cup F$, 
\item  $A_\alpha\cap A_{\alpha'}\subseteq m$ for any distinct $\alpha, \alpha'\in a_\xi'\cup b_\xi'$,
\item  $(a_\xi'\cup b_\xi')\cap (a_\eta'\cup b_\eta')=\emptyset=a_\xi'\cap b_\xi'$.
\end{itemize}

\end{lemma}
\begin{proof}

By Definition \ref{def-P} there are $a_\xi, b_\xi\in [\A]^{<\omega}$ and $m_\xi\in \N$ and disjoint $E_\xi, F_\xi\subseteq m_\xi$
such that 
$$p_\xi=((\bigcup a_\xi\setminus m_\xi)\cup E_\xi, (\bigcup b_\xi\setminus m_\xi)\cup F_\xi).$$
  Moreover, necessarily $a_\xi\cap b_\xi=\emptyset$. By increasing $m_\xi$ and modifying $E_\xi$s and $F_\xi$s we may assume that
$A_\alpha\cap A_{\alpha'}\subseteq m_\xi$ for all distinct $\alpha, \alpha'\in a_\xi\cup b_\xi$.

By passing to a subset of cardinality $\kappa$ we may assume that there are $k', l', m\in\N$ and  $E, F, G, H\in [\N]^{<\omega}$ such that 
$m=m_\xi$, $E=E_\xi$, $F=F_\xi$ for all $\xi<\kappa$
 and $|a_\xi|=k'$ and $|b_\xi|=l'$
for any $\xi\in \kappa$. 

If either of the sets $\{a_\xi: \xi<\kappa\}$ or $\{b_\xi: \xi<\kappa\}$ is of cardinality less than $\kappa$ we may find
 $\Gamma\subseteq\kappa$
of cardinality $\kappa$ such that the set has only one element which means  that the entire $\{p_\xi: \xi\in \Gamma\}$
 is pairwise compatible.
So $p_\xi'=(\emptyset, \emptyset)$ works with $E=F=a_\xi'=b_\xi'=\emptyset$ and $k, l, m=0$.

Otherwise we may assume that all $a_\xi$'s and $b_\xi$'s are distinct.
Using the Delta System Lemma  for families of finite sets of cardinality $\kappa$ by  passing to an uncountable set
 we may assume that
$\{a_\xi: \xi<\omega_1\}$ and $\{b_\xi: \xi<\omega_1\}$ form $\Delta$-systems with roots $a, b\in [\A]^{<\omega}$ respectively.

Let $a_\xi'=a_\xi\setminus a$, $b_\xi'=b_\xi\setminus b$ and $k=|a_\xi'|$, $l=|b_\xi'|$.  Using the Delta System Lemma
again for $\{(a_\xi'\cup b_\xi'): \xi<\kappa\}$ we find $\Gamma\subseteq \kappa$ of cardinality $\kappa$ such that additionally
$(a_\xi'\cup b_\xi')\cap (a_\eta'\cup b_\eta')=\emptyset$ for every distinct $\xi, \eta\in \Gamma$. 
For $\xi\in \Gamma$ we put 
$$p_\xi'=((\bigcup a_\xi'\setminus m)\cup E,  (\bigcup b_\xi'\setminus m)\cup F).$$

Now $p_\xi$ is incompatible with $p_\eta$ if and only if
$$((\bigcup a_\xi\setminus m)\cap (\bigcup b_\eta\setminus m))\cap 
((\bigcup a_\eta\setminus m)\cap (\bigcup b_\xi\setminus m))\not=\emptyset.$$
And $p_\xi'$ is incompatible with $p_\eta'$ if and only if
$$((\bigcup a_\xi'\setminus m)\cap (\bigcup b_\eta'\setminus m))\cap 
((\bigcup a_\eta'\setminus m)\cap (\bigcup b_\xi'\setminus m))\not=\emptyset.$$
But these two conditions are equivalent
since $\bigcup a\cap \bigcup b\setminus m=\emptyset$ and 
$((\bigcup a\cup \bigcup b)\setminus m)\cap ((\bigcup a_\xi'\cup \bigcup b_\xi')\setminus m)=\emptyset$
and $((\bigcup a\cup \bigcup b)\setminus m)\cap ((\bigcup a_\eta'\cup \bigcup b_\eta')\setminus m)=\emptyset$
for any $\xi, \eta\in \Gamma$.
\end{proof}

\section{The splitting partial order $\PP_\A$ and the combinatorial properties of $\A$}
\subsection{Luzin properties of $\A$}

In this subsection we will use two ``Luzin properties'' of $\ad$-families expressed in the following two definitions.

\begin{definition}\label{def-L}  An $\ad$-family $\A$ is called a Luzin family if $\A=\{A_\xi: \xi<\omega_1\}$ and
$f_\eta:\eta\rightarrow \N$ is finite-to-one for each $\eta<\omega_1$, where $f_\eta(\xi)=\max(A_\xi\cap A_\eta)$
for each $\xi<\eta$.
\end{definition}

\begin{definition}[\cite{n-luzin}]\label{def-n-L} Let $n\in \N$ and $\A$ be an $\ad$-family. 
Let $\B_i=\{B^i_\alpha: \alpha<\omega_1\}$
 be disjoint subfamilies of $\A$ for $i < n$. We say that $(\B_i : i<n)$ is an $n$-Luzin gap if there is $m\in \N$ such that
\begin{enumerate}
\item $B^i_\alpha\cap B^j_\alpha\subseteq m$ for all $i<j<n$ and $\alpha<\omega_1$ and
\item $\bigcup_{i\not=j}(B^i_\alpha\cap B^j_\beta)\not\subseteq m$ for all distinct $\alpha, \beta<\omega_1$.
\end{enumerate}
We say that $\A$ contains an $n$-Luzin gap if there is an $n$-Luzin gap $(\B_i : i < n)$ where each $\B_i$ is a subfamily of $\A$ .
\end{definition}

The results of this subsection  clarify the relations among the above Luzin properties.
The existence of an $n$-Luzin gap in an $\ad$-family $\A$ implies the existence
of a $k$-Luzin gap in $\A$ for $k>n>1$ (Proposition \ref{monotone-L}). Being a Luzin family
is much stronger, as one can choose $n$-Luzin gaps in such a family within any $n$ many
uncountable subfamiles (Proposition \ref{hered-L}) and there are in ZFC families admitting
a $2$-Luzin gap, while not containing any Luzin subfamilies (Proposition \ref{juris-L}). 
It is shown in \cite{n-luzin} (Proposition 2.14)
that it is consistent that there are $\ad$-families which contain $3$-Luzin gaps but no $2$-Luzin gaps. 
This will also follow from several our results in the following sections.
However OCA implies that an $\ad$-family contains $n$-Luzin gap if and only if it contains a $k$-Luzin gap
for any $n, k\in \N\setminus\{0,1\}$ (proposition \ref{oca-L} which is 
implicitly included in \cite{n-luzin}: Lemma 2.2 and Proposition 2.9)

\begin{proposition}\label{hered-L} Let $\A$ be a Luzin $\ad$-family and let $\mathcal C_i\subseteq \A$
 be uncountable for each $i<n\in \N\setminus\{0,1\}$.
Then for each $i<n$ there are uncountable $\B_i\subseteq \mathcal C_i$ such that $(\B_i : i<n)$ an n-Luzin gap. 
\end{proposition}
\begin{proof} 
By passing to uncountable subsets we may assume that the $\mathcal C_i$s are pairwise disjoint.  Thinning out
further  for each $i<n$ we may find an uncountable $\B_i\subseteq \mathcal C_i$ such that (1) of Definition
\ref{def-n-L} is satisfied for some $m$.  Let $(\xi_\alpha^i: \alpha<\omega_1)$ be such that
$\B_i=\{A_{\xi_\alpha^i}: \alpha<\omega_1\}$ and $\xi_\alpha^i<\xi_\beta^j$ for each $i, j<n$ and $\alpha<\beta<\omega_1$.
Using Definition \ref{def-L} we obtain a regressive function $f:\omega_1\rightarrow\omega_1$ such that 
for every limit $\beta\in\omega_1$ we have $\max(A_{\xi_\alpha}^i\cap A_{\xi_\beta}^j)>m$
for every $\alpha\in(f(\beta), \beta)$ and every $i, j<n$. By the Pressing Down Lemma we 
obtain an uncountable $\Gamma\subseteq\omega_1$ such that $A_{\xi_\alpha}^i\cap A_{\xi_\beta}^j\not\subseteq m$
for every  $i, j<n$ and every distinct $\alpha,\beta\in \Gamma$ as required in (2) of Definition \ref{def-n-L}.
\end{proof}

\begin{proposition}\label{juris-L}
There is an almost disjoint family $\A$ of subsets of $\N$ of cardinality $\mathfrak c$
such that $\A$ contains a $2$-Luzin gap but does not contain a Luzin subfamily.
\end{proposition}
\begin{proof} It is enough to construct such a family for $\kappa=\mathfrak c$.
This construction is due to Juris Steprans. For every $x\in 2^\N$  consider
$A_x^0=\{t\in 2^{<\omega}: t^\frown 0 \subseteq x\}$ and 
$A_x^1=\{t\in 2^{<\omega}: t^\frown 1\subseteq x\}$.
Note that $A_x^0\cap A^1_x=\emptyset $ and $\A=\{A_x^0, A^1_x: x\in 2^\N\}\subseteq \wp(2^{<\omega})$ is almost disjoint. 
Of course we can identify $2^{<\omega}$ with $\N$. We claim that $(\A_0, \A_1)$ is  a $2$-Luzin gap as witnessed by $m=0$, where
$\A_i=\{A_x^i: x\in 2^\N\}$ for $i\in\{0,1\}$.

  Given distinct $x, y\in 2^\N$ there is
$t\in 2^{<\omega}$ such that $t\subseteq x, y$ and $t^\frown i\subseteq x\setminus y$ 
and $t^\frown(1- i)\subseteq y\setminus x$ for some $i\in\{0,1\}$.
Then $t\in A^i_x\cap A^{(1-i)}_y$, so (2) of Definition \ref{def-n-L} is satisfied as well.

To see that $\A$ contains no Luzin subfamily, consider any uncountable $\B\subseteq \A$.
By passing to an uncountable subfamily we may assume that $\B\subseteq \A_i$ for some $i\in \{0,1\}$.
Now, note that given an uncountable set of reals, there are two uncountable subsets of it included 
in disjoint open sets. This proves that the condition of Definition \ref{def-L} fails
for such $\B$.

\end{proof}

\begin{proposition}\label{monotone-L} Suppose that $\A$ is an $\ad$-family  and $n<k$ are elements of $\N\setminus\{0,1\}$.
If $\A$ contains an $n$-Luzin gap, then it contains a $k$-Luzin gap.
\end{proposition}
\begin{proof}
Let $(\B_i : i<n)$ for $\B_i=\{B^i_\alpha:\alpha<\omega_1\}$ be  an $n$-Luzin gap as witnessed by $m\in \N$.  
 First let us make an observation that
actually any $m'\geq m$ witnesses it for $(\B_i ' : i<n)$ where $\B_i'=\{B^i_\alpha: \alpha\in \Gamma\}$ 
for some uncountable $\Gamma\subseteq\omega_1$.
Indeed, the clause (1) of Definition \ref{def-n-L} is clear for such $m'$ and any uncountable $\Gamma\subseteq\omega_1$.
To obtain (2) find an uncountable $\Gamma\subseteq\omega_1$ such that
$B^i_\alpha\cap m'=B^i_\beta\cap m'$ for all $\alpha, \beta\in \Gamma$ and any $i<n$. Then whenever $\alpha, \beta$ are distinct
elements of $\Gamma$ and $i, j<n$ are distinct,  
by (1) we have $B^i_\alpha\cap B^j_\beta\cap m'=B^i_\alpha\cap B^j_\alpha\cap m'\subseteq  m$,
hence by (2) for $(\B_i : i<n)$ we have $B^i_\alpha\cap B^j_\beta\not\subseteq m'$ for some distinct $i, j<n$.

Now to prove the proposition, by passing to a subset we may assume that
$\A\setminus \bigcup\{\B_i: i<n\}$ is uncountable, so we can find $\B_j\subseteq \A\setminus\bigcup\{\B_i: i<n\}$
for $n\leq j<k$ such that $\B_i\cap \B_j=\emptyset$ for $i<j<k$. Thinning out further we may assume that there is 
$m'>m$ such that  $B^i_\alpha\cap B^j_\alpha\subseteq m'$ for all $i<j<n$ and $\alpha<\omega_1$.
Passing to the uncountable $\Gamma\subseteq\omega_1$ from our initial observation we obtain
a $k$-Luzin gap.

\end{proof}

\begin{proposition}\label{oca-L} Assume {\sf OCA}.  Suppose that $\A$ is an $\ad$-family  
and $n<k$ are elements of $\N\setminus\{0,1\}$.
$\A$ contains an $n$-Luzin gap if and only if $\A$ contains a $k$-Luzin gap.
\end{proposition}
\begin{proof}  First let us note that the following version of {\sf OCA} follows from the standard one: 
Given $l\in\N\setminus\{0\}$  and 
a cover $K_0\cup\dots\cup K_l$ of $X\subseteq 2^\N$ such that $\{(x, y): \{x, y\}\in K_{l'}\}$ is open  in $X^2$ for any $l'\leq l$ there
is $l'\leq l$ and an uncountable   $Y\subseteq X$ which is $l'$-homogeneous. This should be clear,
one needs to apply iterating the standard OCA obtaining either an uncountable homogeneous $Y\subseteq X$ in some 
part $K_{l'}$ or obtaining an uncountable set $Y\subseteq X$ whose pairs avoid $K_{l'}$. 
In the first case we are done, in the second case we work with $Y$ instead of $X$. The remaining colors 
cover the pairs of $Y$ and $\{(x, y): \{x, y\}\in K_{l''}\cap Y\}$ are open  in $Y^2$
or $\l''\leq l$ distinct than $l'$. This procedure can be continued until
the desired set is found.

Let $(\B_i : i<n)$ for $\B_i=\{B^i_\alpha:\alpha<\omega_1\}$ be  a $k$-Luzin gap as witnessed by $m\in \N$.   
For
$$X=\{x_\alpha=(1_{B_\alpha^0},\dots, 1_{B_\alpha^{k-1}})\in (2^\N)^k: \alpha<\omega_1\}$$
define a cover of $[X]^2$ by parts $\{K_{\{i, j\}}: \{i, j\}\in [k]^2\}$ by declaring that  if $\{x_\alpha, x_\beta\}$ belongs to
a $K_{\{i, j\}}$  then
$(B^i_\alpha\cap B^j_\beta)\cup(B^i_\alpha\cap B^j_\beta)\not\subseteq m$. Since  $(\B_i : i<n)$ is assumed to be
$k$-Luzin, there is always such an $\{i, j\}$. This proves that this is really a cover.  Moreover
$\{(x_\alpha, x_\beta): \{x_\alpha, x_\beta\}\in K_{\{i, j\}}\}$ is open for each $\{i, j\}\in [k]^2$.
So by our version of {\sf OCA} we obtain an uncountable $\Gamma\subseteq\omega_1$  and $i'<j'<k$ such that
$(B^{i'}_\alpha\cap B^{j'}_\beta)\cup(B^{i'}_\alpha\cap B^{j'}_\beta)\not\subseteq m$ for every distinct $\alpha, \beta\in \Gamma$
and consequently $(\B_{i'}, \B_{j'})$ is a $2$-Luzin gap. 

It follows from Proposition \ref{monotone-L} that
$\A$ contains an $n$-Luzin gap for any $n\geq 2$.
\end{proof}

\subsection{The c.c.c. of $\PP_\A$ and $n$-Luzin gaps in  $\A$}

In this subsection we analyze the relations between the existence of $n$-Luzin gaps
in an $\ad$-family $\A$ and
the countable chain condition of the partial order $\PP_\A$.
A $2$-Luzin gap in $\A$ yields an uncountable antichain in $\PP_\A$ which yields
an $n$-Luzin gap in $\A$ (Propositions \ref{2L-notccc}, \ref{notccc-nL}).  But none of the implications
can be reversed in {\sf ZFC} (Proposition \ref{ccc3L}, \ref{notcccnot2L}) and $2$-Luzin cannot be weakened
 in {\sf ZFC} to $3$-Luzin in Proposition
\ref{2L-notccc} (Proposition \ref{ccc3L}). However
{\sf OCA} implies that the c.c.c. of $\PP_\A$ is equivalent to the existence
of a $n$-Luzin gap in $\A$ for some $n\in \N$ (Proposition \ref{oca-Lccc}).

\begin{proposition}\label{2L-notccc} Suppose that $\A$ is an $\ad$-family . 
If $\A$ contains a $2$-Luzin gap,  then $\PP_\A$  fails to satisfy c.c.c. In particular, 
$\PP_\A$ fails to satisfy c.c.c. if $\A$ is a Luzin family.
\end{proposition}
\begin{proof}
Let $\{\B_1, \B_2\}$ be a $2$-Luzin gap contained by $\A$. Then 
$\C=\{(B^0_\alpha\setminus m, B^1_\alpha\setminus m): \alpha<\omega_1\}\subseteq \PP_\A$ by (1)
of Definition \ref{def-n-L}. By (2) of Definition \ref{def-n-L} either 
$(B^0_\alpha\setminus m)\cap (B^1_\alpha\setminus m)\not=\emptyset$ or 
$(B^0_\alpha\setminus m)\cap (B^1_\alpha\setminus m)\not=\emptyset$ which means that
$\C$ is an uncountable antichain of $\PP_\A$.

For the proof of the last part of the proposition apply Proposition \ref{hered-L}.
\end{proof}

\begin{proposition}\label{notccc-nL} Suppose that $\A$ is an $\ad$-family.  If  $\PP_\A$  fails to be c.c.c., 
then $\A$ contains an $n$-Luzin gap for some $n\in\N$.
\end{proposition}
\begin{proof}
Let $\{p_\xi: \xi<\omega_1\}$ be an uncountable antichain in $\PP_\A$. By Lemma \ref{thinning} we
may assume that  there are $k, l, m\in \N$ and $a_\xi\in [\A]^{k}$, $b_\xi\in [\A]^{l}$ and
$E, F\subseteq m$ such that
$p_\xi$s are as  $p_\xi'$s in Lemma \ref{thinning}. Let $a_\xi=\{A_\xi^i: i<k\}$ and $b_\xi=\{A_\xi^i: k\leq i<k+l\}$
for $\xi<\kappa$.
By passing to an uncountable subset we may assume that
$A_\xi^i\cap m=A_\eta^i\cap m$   for all $\xi<\eta<\omega_1$ .

 Put  $n=k+l$ and $\A_i=\{A_\xi^i: \xi<\omega\}$ and note that $(\A_i: i<n)$ is an $n$-Luzin gap as witnessed by $m$: condition (1)
of Definition \ref{def-n-L} is clear from the choice of $m$ based on Lemma \ref{thinning} and
condition (2) follows from the fact that by the incompatibility of $p_\xi$ and $p_\eta$
for each $\xi<\eta<\omega_1$ we have $i<k\leq j<n$ such that
either $((A_\xi^i\setminus m)\cup E)\cap ((A_\eta^j\setminus m)\cup F)\not=\emptyset$
or  $((A_\eta^i\setminus  m)\setminus E)\cap ((A_\xi^j\setminus m)\cup F)\not=\emptyset$; however
the fact that $E, F\subseteq m$ are disjoint  implies  that the incompatibility of $p_\xi$ and $p_\eta$ is  equivalent to
the alternative of $A_\xi^i\cap A_\eta^j\not\subseteq m$ or $A_\eta^i\cap A_\xi^j\not\subseteq m$ respectively.
\end{proof}

\begin{proposition}\label{oca-Lccc} Assume {\sf OCA}. Suppose that $\A$ is an $\ad$-family.
$\PP_\A$ is c.c.c. if and only if $\A$ contains an $n$-Luzin gap for every (some)
$n\in \N\setminus\{0,1\}$.
\end{proposition}
\begin{proof} Apply Propositions \ref{oca-L}, \ref{2L-notccc}, \ref{notccc-nL}.
\end{proof}

\begin{proposition}\label{ccc3L} There is a $\sigma$-centered forcing notion $\PP$ such that $\PP$ forces
that there is  an $\ad$-family $\A$  such that $\PP_\A$  satisfies  c.c.c.
and  $\A$ contains a $3$-Luzin gap.
\end{proposition}
\begin{proof}
We consider a forcing notion $\PP$ consisting of conditions 
$$p=(n_p, a_p, (A^0_p(\xi), A^1_p(\xi), A^2_p(\xi): \xi\in a_p))$$ satisfying
\begin{enumerate}
\item $n_p\in \N$, $a_p\in [\omega_1]^{<\omega}$,
\item $A_p^i(\xi)\subseteq n_p$ for $i<3$ and $\xi\in a_p$,
\item  $A_p^i(\xi)\cap A_p^j(\xi)=\emptyset$ for all $\xi\in a_p$ and $i<j<3$,
\item  for any distinct $\xi,\eta\in a_p$ there are distinct $i, j<3$ such that 
$$A_p^i(\xi)\cap A_p^j(\eta)\not=\emptyset.$$
\end{enumerate}
For $p, q\in \PP$ we say that $p\leq q$ if
\begin{enumerate}
 \item[(5)] $n_p\geq n_q$, $a_p\supseteq a_q$,
 \item[(6)] $A_p^i(\xi)\cap n_q=A_q^i(\xi)\cap n_q$ for all $i<3$ and $\xi\in a_q$,
 \item[(7)] $(\bigcup_{i<3} A_p^i(\xi))\cap (\bigcup_{i<3} A_p^i(\eta))\subseteq n_q$ for any distinct $\xi, \eta\in a_q$.
 \end{enumerate}
 First we will prove that if  $p, q\in \PP$ and $n_p=n_q$ and $A^i_p(\xi)=A^i_q(\xi)$ for all $i<3$ and $\xi\in a_p\cap a_q$, then
 $p$ and $q$ are compatible in $\PP$. This and the fact that $X^{\omega_1}$
  with the product topology  is separable for any finite $X$  imply
 that $\PP$ is $\sigma$-centered.
 
For this first note that if $a_q\subseteq a_p$ (and the above hypotheses hold), then $p\leq q$. 
So we may assume that $a_p\setminus a_q\not=\emptyset
 \not=a_q\setminus a_p$.
 To construct $r\in \PP$ which is stronger than $p$ and $q$ define $n=n_p=n_q$ and $a_r=a_p\cup a_q$.
 Let $k=|a_p\setminus q_q|\cdot|a_q\setminus a_p|$ and put $n_r=n+k$. Let $k\setminus n=\{k_{\xi, \eta}: \xi\in a_p\setminus a_q,
 \eta\in a_q\setminus a_p\}$.  Define
 $$
A_r^i(\xi)=
  \begin{cases}
    A_p^i(\xi)= A_q^i(\xi)& \text{if $\xi\in a_p\cap a_q$ and $i=0,1,2$} \\
    A_p^i(\xi)\cup\{k_{\xi,\eta}:\eta\in a_q\setminus a_p\}& \text{if $\xi\in a_p\setminus a_q$ and $i=0$}\\
    A_q^i(\xi)\cup\{k_{\xi,\eta}:\xi\in a_p\setminus a_q\}& \text{if $\xi\in a_q\setminus a_q$ and $i=1$}\\
    A_p^i(\xi)& \text{if $\xi\in a_p\setminus a_q$ and $i=1,2$}\\
    A_q^i(\xi)& \text{if $\xi\in a_q\setminus a_p$ and $i=0,2$}
  \end{cases}
 $$
 Then it is clear that $r\in \PP$ because $k_{\xi, \eta}\in A_r^0(\xi)\cap A_r^1(\eta)$ for $\xi\in a_p\setminus a_q$
  and $\eta\in a_q\setminus a_p$.
 Moreover $r\leq p,q$ because all $k_{\xi, \eta}$s are distinct. This completes the proof that  $\PP$ is $\sigma$-centered.
 
 Let $\dot{\mathbb G}$ be a $\PP$-name for a generic filter in $\PP$.  For $i<3$ and $\xi<\omega_1$
 let $\dot A_\xi^i$ be a $\PP$-name for a subset of $\N$
 such that $\PP$ forces that $\dot A_\xi^i=\bigcup\{A_p^i(\xi): p\in \dot{\mathbb G}\}$.
 Standard density arguments imply that $\PP$ forces that $\dot A_\xi^i$ is an infinite subset of $\N$
 and that 
 $$\A_{\dot{\mathbb G}}=\{\dot A_\xi^i: i<3, \xi<\omega_1\}$$
 is an $\ad$-family. We will prove that $\A_{\dot G}$ is the reqiuired family. 
 Let $\A_{\dot{\mathbb G}}^i=\{\dot A_\xi^i:  \xi<\omega_1\}$ for $i<3$.   
 By a standard density argument  and conditions (3) and (4) above $\PP$ forces that
   $(\A_{\dot{\mathbb G}}^0, \A_{\dot{\mathbb G}}^1, \A_{\dot{\mathbb G}}^2)$
 is a $3$-Luzin gap as witnessed by $m=0$.  The rest of the proof is devoted to proving that 
 $\PP_{\A_{\dot{\mathbb G}}}$ satisfies the c.c.c.
 
 Let $\{\dot \rho_\alpha: \alpha<\omega_1\}$ be $\PP$-names for elements of an uncountable subset of
  $\PP_{\A_{\dot{\mathbb G}}}$.
 By Lemma \ref{thinning} and by passing to a further uncountable subset we may assume that there are conditions
 $p_\alpha\in \PP_\A$ for $\alpha<\omega_1$ and  $k, l, m\in \N$ and disjoint $a_\alpha\in [\omega_1\times 3]^k$, 
 $b_\alpha\in [\omega_1\times 3]^l$
 and $E, F\subseteq m$ such that
 $$p_\alpha\forces \dot\rho_\alpha=\Big((\bigcup_{(\xi, i)\in \check a_\alpha}\dot A_\xi^i\setminus\check m)\cup\check E,
 (\bigcup_{(\xi, i)\in \check b_\alpha}\dot A_\xi^i\setminus\check m)\cup\check F\Big).\leqno (*)$$
 By a standard density arguments and passing to a further uncountable subset we may assume that  for 
 each $\alpha<\omega_1$ we have
 $n=n_{p_\alpha}\geq m$ and $a_\alpha\cup b_\alpha\subseteq a_{p_\alpha}$.  Moreover by another thinning out
 we may assume that there are bijections $\phi_{\beta, \alpha}: a_{p_\alpha}\rightarrow a_{p_\beta}$ such that for every
 $\alpha<\beta<\omega_1$ and $i\in 3$ we have
 \begin{itemize}
 \item $\phi_{\beta, \alpha}(\xi)=\xi$ for $\xi\in a_{p_\alpha}\cap a_{p_\beta}$,
 \item $A_{p_\alpha}^i(\xi)=A_{p_\beta}^i(\phi_{\beta, \alpha}(\xi))$ for  $\xi\in a_{p_\alpha}$,
\item $(\xi, i)\in a_\alpha$ if and only if 
 $(\phi_{\beta, \alpha}(\xi), i)\in a_\beta$ for $\xi\in a_{p_\alpha}$ and for $i<3$,
\item $(\xi, i)\in b_\alpha$ if and only if 
 $(\phi_{\beta, \alpha}(\xi), i)\in b_\beta$ for $\xi\in a_{p_\alpha}$ and for $i<3$.
 \end{itemize}
Pick any two $\alpha<\beta<\omega_1$. Put $p=p_\alpha$ and $q=p_\beta$. We will construct $r\leq p, q$
such that $r$ forces that $\rho_\alpha$ and $\rho_\beta$ are compatible in $\PP_{\A_{\dot G}}$.
Put $a_r=a_p\cup a_q$ and $n_r=n+n'$ where $n'=|a_p\setminus a_q|=|a_q\setminus a_p|$. To define 
$A_r^i(\xi)$s for $\xi\in a_r$ and $i<3$ we  need to fix $\psi: a_p\setminus a_q\rightarrow [3]^2$
such that if $\psi(\xi)=\{i, j\}$, then  either $a_\alpha\cap\{(\xi, i), (\xi, j)\}=\emptyset$ or $b_\alpha\cap\{(\xi, i), (\xi, j)\}=\emptyset$.
Such $\psi(\xi)$ can be found because the complements of $a_\alpha$ and of $b_\alpha$ cover $a_{p_\alpha}\times 3$
as $a_\alpha\cap b_\alpha=\emptyset$ and so  out of three elements $(\xi,0), (\xi, 1), (\xi,2)$ two must be
in the same complement. Note that by the choice of $\alpha$ and $\beta$ we also
have $a_\alpha\cap\{(\xi, i), (\xi, j)\}=\emptyset$ if and only if
$a_\beta\cap\{(\phi(\xi), i), (\phi(\xi), j)\}=\emptyset$ and
$b_\alpha\cap\{(\xi, i), (\xi, j)\}=\emptyset$ if and only if
$b_\beta\cap\{(\phi(\xi), i), (\phi(\xi), j)\}=\emptyset$. Let $a_p\setminus a_q=\{\xi_{k'}: k'<n'\}$ 
and $\phi(\xi_{k'})=\eta_{k'}$ for $k'<n'$.
We are ready to define 
 $$
A_r^i(\xi)=
  \begin{cases}
    A_p^i(\xi)= A_q^i(\xi)& \text{if $\xi\in a_p\cap a_q$ and $i=0,1,2$}, \\
    A_p^i(\xi)\cup\{n+k'\}& \text{if $\xi=\xi_{k'}$ and $i=\min(\psi(\xi))$},\\
    A_p^i(\xi)& \text{if $\xi\in a_p\setminus a_q$ and  $i\not=\min(\psi(\xi))$},\\
    A_q^i(\xi)\cup\{n+k'\}& \text{if $\xi=\eta_{k'}$ and $i=\max(\psi(\xi))$},\\
    A_q^i(\xi)& \text{if $\xi\in a_q\setminus a_p$ and  $i\not=\min(\psi(\phi^{-1}(\xi)))$}.
  \end{cases}
 $$
 It should be clear that $r\in \PP$ and $r\leq p, q$. The remaining part is devoted to proving that
 $r$ forces that $\rho_\alpha$ and $\rho_\beta$ are compatible in $\PP_{\A_{\dot G}}$.
 By (*) it is enough to prove that for any $(\xi, i)\in a_\alpha$ and $(\eta, j)\in b_\beta$ we have
 $r\forces \dot A_{\xi}^i\cap \dot A_{\eta}^j\setminus \check m=\emptyset$ and for any $(\xi, i)\in b_\alpha$
  and $(\eta, j)\in a_\beta$ we have
 $r\forces \dot A_{\xi}^i\cap \dot A_{\eta}^j\setminus \check m=\emptyset$. For this by (3) and (7) 
 it is enough to prove that for any $(\xi, i)\in a_\alpha$ and $(\eta, j)\in b_\beta$ we have we have
 $A_r^i({\xi})\cap A_r^j({\eta})\setminus  m=\emptyset$ and for any $(\xi, i)\in b_\alpha$ and $(\eta, j)\in a_\beta$ we have
 $A_r^i({\xi})\cap A_r^j({\eta})\setminus  m=\emptyset$. We have 
  $$[A_r^i({\xi})\cap A_r^j({\eta})\setminus  m]\cap n=
  [A_q^i(\phi_{\beta, \alpha}(\xi))\cap A_q^j({\eta})\setminus  m]\cap n$$
 which must be empty if $(\phi_{\alpha,\beta}(\xi), i)\in a_\beta$ and $(\eta, j)\in b_\beta$
 or if $(\phi_{\alpha, \beta}(\xi), i)\in b_\beta$ and $(\eta, j)\in a_\beta$ because
 $\dot \rho$ is forced to be a condition of $\PP_{\A_{\dot{\mathbb G}}}$.
 This however happens if and only if  $(\xi, i)\in a_\alpha$ and $(\eta, j)\in b_\beta$
 and $(\xi, i)\in b_\alpha$ and $(\eta, j)\in a_\beta$ respectively.
 So we are left with proving $[A_r^i({\xi})\cap A_r^j({\eta})\setminus  m]\cap (n_r\setminus n)=\emptyset$
 under appropriate hypotheses. This follows from the construction, more specifically from
 the choice of $\psi$ which guarantees that 
 $A_r^i(\xi)\cap A_r^j(\eta)$ intersect above $n$ only if $\eta=\phi_{\beta, \alpha}(\xi)$
 and $i=\min(\psi(\xi)$ and $j=\max(\psi(\xi))$. As
 for $\eta=\psi_{\beta, \alpha}(\xi)$ we have that $(\eta, j)\in a_\beta$ ($(\eta, j)\in b_\beta$)
 if and only if $(\xi, j)\in a_\alpha$ ($(\xi, j)\in b_\alpha$), the choice
 of $\psi$ guarantees that the intersection is empty.
\end{proof}

\begin{proposition}\label{notcccnot2L} There is a $\sigma$-centered forcing notion $\PP$ such that $\PP$ forces
that there is  an $\ad$-family $\A$  such that $\PP_\A$  admits an uncountable antichain
and  $\A$ does not contain a $2$-Luzin gap.
\end{proposition}
\begin{proof}
We consider a forcing notion $\PP$ consisting of conditions 
$$p=(n_p, a_p, (A^0_p(\xi), A^1_p(\xi), A^2_p(\xi), A^3_p(\xi): \xi\in a_p))$$ satisfying
\begin{enumerate}
\item $n_p\in \N$, $a_p\in [\omega_1]^{<\omega}$,
\item $A_p^i(\xi)\subseteq n_p$ for $i<4$ and $\xi\in a_p$,
\item  $A_p^i(\xi)\cap A_p^j(\xi)=\emptyset$ for all $\xi\in a_p$ and $i<j<4$,
\item  for any distinct $\xi,\eta\in a_p$ there is $i<3$ such that 
$$[A_p^i(\xi)\cap A_p^3(\eta)]\cup[A_p^3(\xi)\cap A_p^i(\eta)]\not=\emptyset.$$
\end{enumerate}
For $p, q\in \PP$ we say that $p\leq q$ if
\begin{enumerate}
 \item[(5)] $n_p\geq n_q$, $a_p\supseteq a_q$,
 \item[(6)] $A_p^i(\xi)\cap n_q=A_q^i(\xi)\cap n_q$ for all $i<4$ and $\xi\in a_q$,
 \item[(7)] $(\bigcup_{i<4} A_p^i(\xi))\cap (\bigcup_{i<4} A_p^i(\eta))\subseteq n_q$ for any distinct $\xi, \eta\in a_q$.
 \end{enumerate}
 First we will prove that if  $p, q\in \PP$ and $n_p=n_q$ and $A^i_p(\xi)=A^i_q(\xi)$ for all $i<4$ and $\xi\in a_p\cap a_q$, then
 $p$ and $q$ are compatible in $\PP$. This and the fact that $X^{\omega_1}$ with the product topology
   is separable for any finite $X$  imply
 that $\PP$ is $\sigma$-centered.
 
For this first note that if $a_q\subseteq a_p$ (and the above hypotheses hold), then $p\leq q$. 
So we may assume that $a_p\setminus a_q\not=\emptyset
 \not=a_q\setminus a_p$.
 To construct $r\in \PP$ which is stronger than $p$ and $q$ define $n=n_p=n_q$ and $a_r=a_p\cup a_q$.
 Let $k=|a_p\setminus q_q|\cdot|a_q\setminus a_p|$ and put $n_r=n+k$. Let $k\setminus n=\{k_{\xi, \eta}: \xi\in a_p\setminus a_q,
 \eta\in a_q\setminus a_p\}$.  Define
 $$
A_r^i(\xi)=
  \begin{cases}
    A_p^i(\xi)= A_q^i(\xi)& \text{if $\xi\in a_p\cap a_q$ and $i<4$} \\
    A_p^i(\xi)\cup\{k_{\xi,\eta}:\eta\in a_q\setminus a_p\}&
     \text{if $\xi\in a_p\setminus a_q$ and $i=0$}\\
    A_q^i(\xi)\cup\{k_{\xi,\eta}:\xi\in a_p\setminus a_q\}&
     \text{if $\xi\in a_q\setminus a_p$ and $i=3$}\\
    A_p^i(\xi)& \text{if $\xi\in a_p\setminus a_q$ and $i=1, 2, 3$}\\
    A_q^i(\xi)& \text{if $\xi\in a_q\setminus a_p$ and $i=0, 1, 2$}
  \end{cases}
 $$
 Then it is clear that $r\in \PP$ because $k_{\xi, \eta}\in A_r^0(\xi)\cap A_r^3(\eta)$ for $\xi\in a_p\setminus a_q$
  and $\eta\in a_q\setminus a_p$.
 Moreover $r\leq p,q$ because all $k_{\xi, \eta}$s are distinct.
  This completes the proof that  $\PP$ is $\sigma$-centered.
 
 Let $\dot{\mathbb G}$ be a $\PP$-name for a generic filter in $\PP$.  
 For $i<4$ and $\xi<\omega_1$ 
 let $\dot A_\xi^i$ be a $\PP$-name for a subset of $\N$
 such that $\PP$ forces that $\dot A_\xi^i=\bigcup\{A_p^i(\xi): p\in \dot{\mathbb G}\}$.
 Standard density arguments imply that $\PP$ forces that $\dot A_\xi^i$ is an infinite subset of $\N$
 and that 
 $$\A_{\dot{\mathbb G}}=\{\dot A_\xi^i: i<4, \xi<\omega_1\}$$
 is an $\ad$-family. We will prove that $\A_{\dot{\mathbb G}}$ is the required family. 
 Let $\dot \rho_\xi=(\bigcup_{i<3}\dot A_\xi^i, \dot A_\xi^3)$ for $\xi<\omega_1$. By a standard density argument  and conditions
  (3) and (4) above $\PP$ forces that  $\{\dot \rho_\xi:\xi<\omega_1)$ is an antichain in $\PP_{\A_{\dot{\mathbb G}}}$.  
 The rest of the proof is devoted to showing that ${\A_{\dot{\mathbb G}}}$ does not admit a $2$-Luzin gap.
 
 Let $\dot\xi_\alpha, \dot\eta_\alpha$ for $\alpha<\omega_1$ be $\PP$-names for elements of $\omega_1$ 
 let $\dot k_\alpha, \dot l_\alpha$ for $\alpha<\omega_1$ be $\PP$-names for elements of $4$ 
 and let $\dot m$ be a $\PP$-name for an element of $\N$ such that $\PP$ forces  that 
\begin{itemize}
\item $\langle \dot \xi_\alpha, \dot k_\alpha \rangle\not=\langle\dot\eta_\alpha, \dot l_\alpha\rangle$
\item  $\{\langle \dot \xi_\alpha, \dot k_\alpha \rangle,\langle\dot\eta_\alpha, \dot l_\alpha\rangle\}
\cap \{\langle \dot \xi_\beta, \dot k_\beta \rangle,\langle\dot\eta_\beta, \dot l_\beta\rangle\}=\emptyset$
\item $\dot A_{\xi_\alpha}^{\dot k_\alpha}\cap \dot A_{\eta_\alpha}^{\dot l_\alpha}\subseteq \dot m$
\end{itemize}
 for every $\alpha<\beta<\omega_1$.
 It is enough to show that for every $p\in \PP$ there is  $\alpha<\beta<\omega_1$ and $q\leq p$ which
 forces that 
 $$A_{\dot\xi_\alpha}^{\dot k_\alpha}\cap A_{\dot\eta_\beta}^{\dot l_\beta},
 A_{\dot\eta_\alpha}^{\dot l_\alpha}\cap A_{\dot\xi_\beta}^{\dot k_\beta}\subseteq \dot m\leqno (*)$$
  Below any condition of $\PP$  there are conditions $p_\alpha$ for $\alpha<\omega_1$ 
  and $\xi_\alpha, \eta_\alpha<\omega_1$ and $k_\alpha, l_\alpha, m\in \N$ 
  such that for each $\alpha$ the condition  $p_\alpha$ forces that
    $\dot\xi_\alpha=\check\xi_\alpha, \dot \eta_\alpha=\check \eta_\alpha,
     \dot k_\alpha=\check k_\alpha,  \dot l_\alpha=\check l_\alpha, \dot m=\check m$ (note that
     $\xi_\alpha$ may be equal to $\eta_\alpha$ or $k_\alpha$ may be equal to $l_\alpha$).
  
By a standard density arguments and passing to an uncountable subset we may assume that  for each $\alpha<\omega_1$ we have
 $n=n_{p_\alpha}$, $k_\alpha=k$ and $l_\alpha=l$.
 Moreover by another thinning out
 we may assume that $(a_{p_\alpha}: \alpha<\omega_1)$ forms a $\Delta$-system and that
 there are bijections $\phi_{\beta, \alpha}: a_{p_\alpha}\rightarrow a_{p_\beta}$ such that for every
 $\alpha<\beta<\omega_1$ and $i<4$ we have
 \begin{itemize}
 \item $\phi_{\beta, \alpha}(\xi)=\xi$ for $\xi\in a_{p_\alpha}\cap a_{p_\beta}$,
 \item $A_{p_\alpha}^i(\xi)=A_{p_\beta}^i(\phi(\xi))$ for  $\xi\in a_{p_\alpha}$,
 \item $\phi_{\beta, \alpha}(\xi_\alpha)=\xi_\beta$
 \item $\phi_{\beta, \alpha}(\eta_\alpha)=\eta_\beta$
 \end{itemize}
 Since the forcing is c.c.c. and forces that  
 $\{\langle \dot \xi_\alpha, \dot k_\alpha \rangle,\langle\dot\eta_\alpha, \dot l_\alpha\rangle\}
\cap \{\langle \dot \xi_\beta, \dot k_\beta \rangle,\langle\dot\eta_\beta, \dot l_\beta\rangle\}=\emptyset$
we can conclude that $\xi_\alpha, \eta_\alpha\in a_{p_\alpha}\setminus a_{p_\beta}$
for any distinct $\alpha, \beta<\omega_1$.

Pick any two $\alpha<\beta<\omega_1$. Put $p=p_\alpha$ and $q=p_\beta$. We will construct $r\leq p, q$
such that $r$ forces (*). Let $j<3$ be different than $k, l$.
To construct $r\in \PP$ which is stronger than $p$ and $q$ define  $a_r=a_p\cup a_q$.
 Let $k=|a_p\setminus q_q|\cdot|a_q\setminus a_p|$ and put $n_r=n+k$. Let $k\setminus n=\{k_{\xi, \eta}: \xi\in a_p\setminus a_q,
 \eta\in a_q\setminus a_p\}$.  Define
 $$
A_r^i(\xi)=
  \begin{cases}
    A_p^i(\xi)= A_q^i(\xi)& \text{if $\xi\in a_p\cap a_q$ and $i<4$} \\
    A_p^i(\xi)\cup\{k_{\xi,\eta}:\eta\in a_q\setminus a_p\}& \text{if $\xi\in a_p\setminus a_q$ and $i=j$}\\
    A_q^i(\xi)\cup\{k_{\xi,\eta}:\xi\in a_p\setminus a_q\}& \text{if $\xi\in a_q\setminus a_q$ and $i=3$}\\
    A_p^i(\xi)& \text{if $\xi\in a_p\setminus a_q$ and $i\not=j$}\\
    A_q^i(\xi)& \text{if $\xi\in a_q\setminus a_p$ and $i<3$}
  \end{cases}
 $$
 Then it is clear that $r\in \PP$ because $k_{\xi, \eta}\in A_r^j(\xi)\cap A_r^3(\eta)$ for $\xi\in a_p\setminus a_q$
  and $\eta\in a_q\setminus a_p$.
 Moreover $r\leq p,q$ because all $k_{\xi, \eta}$s are distinct.  The remaining part is 
 devoted to proving 
 (*).
By (3) and (7)
it is enough to prove that
$A_r^{k}({\xi_\alpha})\cap A_r^{l}({\eta_\beta}),
 A_r^{l}({\eta_\alpha})\cap A_r^{k}({\xi_\beta})\subseteq m$.
 As $A_r^{k}({\xi_\alpha})\cap n=A_p^{k}({\xi_\alpha})\cap n$ and
$A_r^{l}({\eta_\beta})\cap n=A_q^{l}({\eta_\beta})\cap n=A_p^{l}({\eta_\alpha})\cap n$
we conclude that $A_r^{k}({\xi_\alpha})\cap A_r^{l}({\eta_\beta})\cap \subseteq
A_p^{k}({\xi_\alpha})\cap A_p^{l}({\eta_\alpha})$ which is included in $m$
since $\PP$ forces that 
$\dot A_{\xi_\alpha}^{\dot k_\alpha}\cap \dot A_{\eta_\alpha}^{\dot l_\alpha}\subseteq \dot m$.
Similarly  $A_r^{l}({\eta_\alpha})\cap A_r^{k}({\xi_\beta})\cap n\subseteq m$. 

Now no $k_{\xi,\eta}$ can belong to $A_r^{k}({\xi_\alpha})\cap A_r^{l}({\eta_\beta})$ or
 $A_r^{l}({\eta_\alpha})\cap A_r^{k}({\xi_\beta})$  
 because $j\not\in \{k, l\}$.
\end{proof}

\subsection{Big antichains in $\PP_\A$}
In this subsection we show that  any
Luzin family $\A$ exemplifies a spectacular failure of the c.c.c. of $\PP_\A$, that is they are $L$-families
 (Proposition \ref{anti-centered}. Recall Definition \ref{def-LL} of an $L$-family). Luzin families are just of size $\omega_1$, but
 in {\sf ZFC} there are $\ad$-families $\A$ admitting antichains of size $\mathfrak c$ in $\PP_\A$
 (Proposition \ref{juris-ac}).
 Nevertheless consistently all $\ad$-families  of size $\mathfrak c$ have plenty of big sets of pairwise
 compatible elements (Proposition \ref{sacks}).

\begin{proposition}\label{anti-centered}
 If $\A$ is a Luzin family, it is an $L$-family. 
\end{proposition}
\begin{proof} Let $\A=\{A_\alpha:\alpha<\omega_1\}$.
Let $f_\eta:\eta\rightarrow \N$ be the finite-to-one function which 
witnesses that $\A$ is a Luzin family as in Definition \ref{def-L}.

Let $\{p_\xi: \xi<\omega_1\}$ be essentially distinct and let
and $A_{p_\xi}= (\bigcup a_\xi\setminus m_\xi)\cup F_\xi$ and $B_{p_\xi}= 
(\bigcup b_\xi\setminus m_\xi)\cup G_\xi$ for $\xi<\omega_1$,
where $a_\xi, b_\xi\in [\A]^{<\omega}$ are such that $a_{\xi}\not=a_{\xi'}$ 
and $b_\xi\not=b_{\xi'}$  and $F_\xi, G_\xi\subseteq m_\xi$ are disjoint and $m_\xi\in \N$ for $\xi<\xi'<\omega_1$. 
We may assume that $m_\xi=m$, $F_\xi=F$ and $G_\xi=G$  and $a_\xi\not=\emptyset\not=b_\xi$ for all $\xi<\omega_1$. Let
$$\alpha_\xi=\max(\{\alpha: A_{\alpha}\in a_\xi\}), \  \beta_\xi=\max(\{\alpha: A_{\alpha}\in b_\xi\})$$
for $\xi<\omega_1$.

The fact that $p_\xi$s are essentially distinct ($a_{\xi}\not=a_{\xi'}$ 
and $b_\xi\not=b_{\xi'}$ for $\xi<\xi'<\omega_1$) implies  that for each $\xi<\omega_1$ there may be at most countably
many $\xi'<\omega_1$ such that $\alpha_\xi=\alpha_{\xi'}$ or $\beta_\xi=\beta_{\xi'}$.
It follows that 
there is a function
$g:\omega_1\rightarrow\omega_1$ such that  $ \beta_\xi<g(\alpha_\xi)$
and  $\alpha_\xi<g(\beta_\xi)$
 for any $\xi<\omega_1$. 
 
 Let 
$(\eta_\alpha:\alpha<\omega_1)$ be an enumeration of a club set $C\subseteq\omega_1$ such that
$\beta<\alpha\in C$ implies $g(\beta)<\alpha$.
Let $\PP_\alpha=\{p_\xi: \alpha_\xi, \beta_\xi\in [\eta_\alpha, \eta_{\alpha+1})\}$. $\PP_\alpha$s are pairwise disjoint. The choice
of $C$ implies that $\{p_\xi: \xi<\omega_1\}=\bigcup_{\alpha<\omega_1}\PP_\alpha$ (take $\alpha<\omega_1$
such that $\min(\alpha_\xi,\beta_\xi)\in [\eta_\alpha, \eta_{\alpha+1})$ and note that 
$\max(\alpha_\xi,\beta_\xi)\in [\eta_\alpha, \eta_{\alpha+1})$ as well).  
 $\PP_\alpha$s are also countable.
For $n\in \N$ let $\Q_n\subseteq \{p_\xi: \xi<\omega_1\}$ be such that $|\Q_n\cap \PP_\alpha|\leq 1$
for every $\alpha<\omega_1$ and $\bigcup_{n\in \N}\Q_n=\{p_\xi: \xi<\omega_1\}$.
It is enough to decompose each uncountable $\Q_n$ into countably many subsets which are pairwise incompatible in $\PP_\A$.

Fix $n\in \N$ such that $\Q_n$ is uncountable and define $h_n(p_\xi)\in \N$ for $p_\xi\in \Q_n$ by induction on $\alpha<\omega_1$
such that $p_\xi\in \Q_n\cap \PP_\alpha$.
If $h_n|(\Q_n\cap\bigcup_{\beta<\alpha}\PP_\alpha)$ is already defined and $\Q_n\cap \PP_\alpha\not=\emptyset$
is witnessed by $p_\xi\in \Q_n\cap \PP_\alpha$, then define
$$X_\alpha=\{p_{\xi'}: p_{\xi'}\in\PP_\beta,  \beta<\alpha, \min\big(f_{\alpha_\xi}[[\eta_\beta, \eta_{\beta+1})]\big) \leq m\}.$$
Since $f_{\alpha_\xi}$ is finite-to-one, the set $X_\alpha$ is a finite subset of $\{p_{\xi'}: \xi'<\xi\}$. Now choose
$h_n(p_\xi)$ to be any element of $\N$ not belonging to $h_n[X_\alpha]$.

It follows from the construction of $h_n$ that $h_n(p_\xi)=h_n(p_\xi')$ for $\xi'<\xi$ implies that 
$f_{\alpha_\xi}(\beta_{\xi'})>m$ and consequently that $(A_{\alpha_{\xi}}\cap B_{\beta_{\xi'}})\setminus m\not=\emptyset$
and  so $p_\xi$ and $p_{\xi'}$ are incompatible in $\PP_\A$. This completes the proof as the fibers of $h_n$
yields the required decomposition of $\Q_n$.
\end{proof}

\begin{proposition}\label{juris-ac} There is an almost disjoint
 family $\A$ of subsets of $\N$ of cardinality $\mathfrak c$
such that $\PP_\A$ contains  an antichain of cardinality $\mathfrak c$.
\end{proposition}
\begin{proof} 
This construction is due to Juris Steprans. For every $x\in 2^\N$  consider
$A_x^0=\{t\in 2^{<\omega}: t^\frown 0\subseteq  x\}$ and $A_x^1=\{t\in 2^{<\omega}: t^\frown 1\subseteq x\}$.
Note that $A_x^0\cap A^1_x=\emptyset $ and $\{A_x^0, A^1_x: x\in 2^\N\}\subseteq \wp(2^{<\omega})$ is almost disjoint.
Of course we can identify $2^{<\omega}$ with $\N$.  Given distinct $x, y\in 2^\N$ there is
$t\in 2^{<\omega}$ such that $t\subseteq x, y$ and $t^\frown i\subseteq x$ and $t^\frown(1- i)\subseteq y$ for some $i\in\{0,1\}$.
Then $t\in A^i_x\cap A^{(1-i)}_y$ and so $(A_x^0, A^1_x)$ is incompatible with $(A_y^0, A^1_y)$.
So $\{(A_x^0, A^1_x): x\in 2^\N\}$ is the required antichain.
\end{proof}

\begin{proposition}\label{sacks} It is consistent that for every almost disjoint family of subsets of $\N$ of cardinality $\mathfrak c$
the forcing $\PP_\A$ contains a set of cardinality $\mathfrak c$ of pairwise compatible essentially distinct conditions.
\end{proposition}
\begin{proof} It is proved in \cite{oursacks} that in the iterated Sacks model every almost disjoint family $\A$ of
cardinality $\mathfrak c$ of subsets of $\N$ contains a subfamily $\B$ of cardinality $\mathfrak c$ which is $\R$-embeddable.
By Proposition \ref{Rembed-centered} the forcing $\PP_\B$ contains a pairwise compatible set of cardinality $\mathfrak c$ which is
pairwise compatible set in $\PP_\A$.
\end{proof}

\subsection{The $\sigma$-centeredness of $\PP_\A$ and the $\R$-embeddability of $\A$}

\begin{definition}\label{def-Rembed} An $\ad$-family $\A$ is called $\R$-embeddable if and only if
there is an injection $f:\N\rightarrow \Q$ such that for each $A\in \A$ the set $f[A]$ is the set of all terms of a convergent sequence
 to $r_A\in \R\setminus \Q$
and $r_A\not=r_{A'}$ for all distinct $A, A'\in \A$.
\end{definition}

$\R$-embeddability was first formally defined in \cite{fernando} using a slightly different condition. The equivalence  of these definitions
and several other conditions is proved in Lemma 2 of \cite{oursacks} (using a slightly different language of $\Psi$-spaces).

In this subsection we investigate the relation between the $\R$-embeddability of
an $\ad$-family $\A$ and the property of $\PP_\A$ being $\sigma$-centered. 
If $\A$ id $\R$-embeddable, then $\PP_\A$ is $\sigma$-centered (Proposition \ref{Rembed-centered})
but the reverse implication fails in {\sf ZFC} (Proposition \ref{notRembed-centered}).
We also prove that no maximal $\ad$-family $\A$ yields $\PP_\A$ $\sigma$-centered 
(Proposition \ref{mad-notcentered}).

\begin{proposition}\label{sigmac-char} Suppose that $\A$ is an $\ad$-family.  $\PP_\A$ is $\sigma$-centered if and only if there
 are $A_n, B_n\subseteq \N$ for $n\in \N$ such that
$A_n\cap B_n=\emptyset$ for all $n\in \N$ and for every $p\in \PP_\A$ there is $n\in \N$
satisfying $A_p\subseteq A_n$ and $B_p\subseteq B_n$.
\end{proposition}
\begin{proof} The sufficiency is clear.  For the necessity, let $\PP_\A=\bigcup_{n\in \N}\PP_n$, where each $\PP_n$
is pairwise compatible.
Then $A_n=\bigcup\{A_p: p\in \PP_n\}$ and $B_n=\bigcup\{B_p: p\in \PP_n\}$ work.
\end{proof}

\begin{proposition}\label{Rembed-centered} Suppose that $\A$ is an $\ad$-family. If $\A$ is $\R$-embeddable, 
then $\PP_\A$ is $\sigma$-centered.
\end{proposition}
\begin{proof}
Let $f: \N\rightarrow \Q$ witnesses the $\R$-embeddability of $\A$. Given $p\in \PP_\A$ since $f$ is an injection,
$f[A_p]$ and $f[B_p]$ are disjoint. Moreover both $f[A_p]$ and $f[B_p]$ are subsets of $\Q$ almost equal to finite unions
of sequences converging to irrationals and the irrationals for $f[A_p]$ must be all distinct from the irrationals for
$f[A_p]$. It follows that the closures of $f[A_p]$ and $f[B_p]$ consists of $f[A_p]$ and $f[B_p]$ and the irrational limits respectively,
and so the closures of $f[A_p]$ and $f[B_p]$ are disjoint. It follows that there is $(U, V)\in \B^2$ such that 
$f[A_p]\subseteq U$ and $f[B_p]\subseteq V$, where $\B^2=\{(U_n, V_n): n\in \N\}$ is
the set of all pairs $(U, V)$ such that $U\cap V=\emptyset$ and $U$ and $V$ are finite unions of intervals with rational end-points.
It follows that $A_n=f^{-1}[U_n]$ and $B_n=f^{-1}[V_n]$ satisfy the condition from Proposition \ref{sigmac-char}
and so $\PP_\A$ is $\sigma$-centered.
\end{proof}

\begin{proposition}\label{notRembed-centered} There is an $\ad$-family $\A$  which  is not $\R$-embeddable 
but $\PP_\A$ is $\sigma$-centered.
\end{proposition}
\begin{proof} Let $\B=\{B_\alpha: \alpha<\mathfrak c\}$ be any $\ad$-family which is $\R$-embeddable. 
We will construct an $\ad$-family $\A=\{A_\xi^1, A_\xi^2:\xi<\mathfrak c\}$ as in the proposition such that
 for each $\xi<\mathfrak c$ and
each $i=0,1$ either there 
is $\alpha_i(\xi)<\mathfrak c$ such each $A_\xi^i=B_{\alpha_i(\xi)}$ or there are 
$\beta_i(\xi)<\gamma_i(\xi)<\mathfrak c$ such that
$A_\xi^i=B_{\beta_i(\xi)}\cup B_{\gamma_i(\xi)}$. Note that the above property of elements of $\A$
 implies that $\PP_\A$ is $\sigma$-centered
as long as $\A$ is an $\ad$-family.  
This follows from Proposition \ref{Rembed-centered} is because $\PP_\A\subseteq \PP_\B$ in this case. 

To construct recursively  $\A$ let $\{f_\xi: \xi<\mathfrak c\}$ be an enumeration of all injective functions from $\N$ into $\Q$. 
Suppose that $\{A_\xi^i: \xi<\eta, i=0,1\}$ has been constructed for $\eta<\mathfrak c$.   
Let $X=\mathfrak c\setminus\{\alpha_i(\xi), \beta_i(\xi), \gamma_i(\xi): \xi<\eta\}$. We will pick
the ordinals $\alpha_i(\xi), \beta_i(\xi), \gamma_i(\xi)$ from $X$. This will guarantee that $\A$ is an $\ad$-family.
At stage $\eta$ consider three cases. 
First case is when for each $\alpha\in X$ the set  $f_\xi[B_\alpha]$ 
is the set of all terms of a convergent sequence  and  the limits
$f_\xi[B_\alpha]$ and  $f_\xi[B_{\alpha'}]$ are distinct  for distinct $\alpha, \alpha'\in X$.  
In such a case define $A_\xi^0=B_\alpha\cup B_{\alpha'}$ and $A_\xi^1=B_{\alpha''}$ for any distinct $\alpha, \alpha'\alpha''\in X$.
The second case is when for each $\alpha\in X$ the set  $f_\xi[B_\alpha]$ is the set of all terms of a convergent
 sequence  but  the limits
$f_\xi[B_\alpha]$ and  $f_\xi[B_{\alpha'}]$ are equal for some  distinct $\alpha, \alpha'\in X$.
In such a case define $A_\xi^0=B_\alpha$ and $A_\xi^1= B_{\alpha'}$.
The third case is when there is $\alpha\in X$ such that $f_\xi[B_\alpha]$ is not the set of all terms of a convergent sequence.
In such a case define $A_\xi^0=B_\alpha$ and $A_\xi^1= B_{\alpha'}$ for any $\alpha'\in X\setminus\{\alpha\}$.

It should be clear from the construction that no injection $f_\xi:\N\rightarrow \Q$ witnesses the fact that $\A$ is $\R$-embeddable, 
and hence
$\A$ is as required.
\end{proof}

\begin{proposition}\label{mad-notcentered}  Suppose that $\A$ is a maximal $\ad$-family. Then $\PP_\A$
 is  not $\sigma$-centered.
\end{proposition}
\begin{proof} Let $A_n, B_n\subseteq\N$ be as in Proposition \ref{sigmac-char}. Assume  that $\A$ is maximal.
We will aim at arriving at a contradiction.
By extending we may assume that
$A_n\cup B_n=\N$ for every $n\in \N$.  Define $F:\N\rightarrow \{0,1\}^\N$ by $F(k)(n)=0$ if and only if $k\in A_n$.
For $A\in\A$ let $X_A\subseteq\{0,1\}^\N$ be the set of all accumulation points of $F[A]$. Clearly each $X_A$ is a closed subset
of $\{0,1\}^\N$ for $A\in \A$.
Since $\A$ is an $\ad$-family, if $A, B\in \A$ are distinct, then there is $m\in \N$ such that $A\cap B\subseteq m$ and so  
$(\{A\setminus m\}, \{B\setminus m\})\in \PP_\A$. It follows from
the choice of $A_n, B_n$s  that there is $n\in \N$ such that $F(k)(n)\not =F(k')(n)$ for $k\in A\setminus m$ and $k'\in B\setminus m$
and consequently $X_A\cap X_B=\emptyset$ for any two distinct $A, B\in \A$. 
Similarly by using $(\{A\setminus m\}, \{\{m\}\})$ for $m\in \N$ we conclude that $F[\N]\cap X_A=\emptyset$ for all $\A\in \A$;
and taking $(\{\{m\}\}, \{\{m'\}\})$ for distinct $m, m'\in \N$ we conclude that $F$ is injective, in particular each $X_A$
for $A\in\A$ is nonempty by the compactness of $\{0,1\}^\N$.

By the metrizability of $\{0,1\}^\N$, if there $x\in F_A$ for $A\in \A$ which is in the closure of
$F[\N\setminus A]$, there is a nontrivial subsequence  of $F[\N\setminus A]$ convergent to $x$ (as $F[\N\setminus A]$ is disjoint
from $X_A$). Then by the maximality
of $\A$ there is $B\in \A$ such that $x\in F_B$ which contradicts our finding that $X_A$s are pairwise disjoint for $A\in \A$.
It follows that every point of $F_A$ for $A\in \A$ has an open basic neighborhood $U$ in $\{0,1\}$  such that
$U\cap F_B=\emptyset$ for all $B\in \A\setminus\{A\}$. But this contradicts the fact that
 there are only countably many basic open sets
of $\{0,1\}^\N$ and $\A$ must be uncountable as a maximal $\ad$-family.

\end{proof}

\subsection{Antiramsey families and dichotomies}
In this subsection we prove  a consistent dichotomy (Proposition \ref{oca}):
 $\PP_\A$ is either $\sigma$-centered
or fails to be c.c.c. which will find application in the following sections. 
This dichotomy is only consistent as consistently there exist strong counterexamples to it which
we call antiramsey $\ad$-families (Propositions \ref{antiramsey-gen}, \ref{antiramsey-ch}, 
\ref{antiramsey-cohen}). They also produce new interesting results concerning Banach spaces
described in the following sections.

\begin{proposition}\label{ma} Assume {\sf MA}.  Suppose that $\A$ is an $\ad$-family of cardinality less than $\mathfrak c$. Then 
either $\PP_\A$ does not satisfy  c.c.c. or  $\PP_\A$ is $\sigma$-centered.
\end{proposition}
\begin{proof}
If $|\A|<\mathfrak c$, then $|\PP_\A|<\mathfrak c$. The proposition follows from 
the well-known fact that under {\sf MA} such c.c.c. forcings are $\sigma$-centered.
\end{proof}
\begin{proposition}[{\sf OCA}]\label{oca}  Suppose that $\A$ is $\ad$-family. Then 
either $\PP_\A$ does not satisfy c.c.c. or it $\PP_\A$ is $\sigma$-centered.
\end{proposition}
\begin{proof}
Let 
$$X=\{(1_{A_p}, 1_{B_p}): p\in \PP_A\}\subseteq2^\N\times 2^\N.$$
Following \cite{n-luzin} define a coloring $c:[X]^2\rightarrow\{0,1\}$
 as $c(x, y)=0$ if and only if $x=(1_{A_p}, 1_{B_p})$ for
 $p\in \PP_\A$ and $y=(1_{A_q}, 1_{B_q})$ for
 $q\in \PP_\A$ and $p$ and $q$ are incompatible in $\PP_\A$.
 Note that $c^{-1}[\{0\}]$ is
open in $2^\N\times 2^\N$ because for distinct $x, y\in X$ we have 
$c(x, y)=1$ if and only if there is $n\in \N$ such that
$x_1(n)=1=y_2(n)$ or $y_1(n)=1=x_2(n)$, where $x=(x_1, x_2)$ and $y=(y_1, y_2)$.
So {\sf OCA} can be applied to the coloring $c$ to conclude that
either $X$ can be covered by countably many $1$-monochromatic sets, that is
$\PP_\A$ is $\sigma$-centered or $X$ contains an uncountable
$0$-monochromatic set, that is $\PP_\A$ does  not satisfy the c.c.c. 
\end{proof}

\begin{proposition}[{\sf CH}]\label{antiramsey-gen} For any uncountable regular cardinal $\kappa$
 there is a c.c.c forcing $\PP$ which forces  that there is an antiramsey
$\ad$-family $\A$ of cardinality $\kappa=\mathfrak c$ which is a maximal $\ad$-family.
\end{proposition}
\begin{proof} Assume  {\sf CH} and choose a regular uncountable $\kappa$.
We consider a forcing notion $\PP$ consisting of conditions 
$$p=(n_p, a_p, (A_p(\xi): \xi\in a_p))$$ satisfying
\begin{enumerate}
\item $n_p\in \N$, $a_p\in [\kappa]^{<\omega}$,
\item $A_p(\xi)\subseteq n_p$ for  $\xi\in a_p$,
\end{enumerate}
For $p, q\in \PP$ we say that $p\leq q$ if
\begin{enumerate}
 \item[(5)] $n_p\geq n_q$, $a_p\supseteq a_q$,
 \item[(6)] $A_p(\xi)\cap n_q=A_q(\xi)\cap n_q$ for all  $\xi\in a_q$,
 \item[(7)] $A_p(\xi))\cap A_p(\eta)\subseteq n_q$ for any distinct $\xi, \eta\in a_q$.
 \end{enumerate}
 First note  that if  $p, q\in \PP$ and $n_p=n_q$ and $A_p(\xi)=A_q(\xi)$ for all $\xi\in a_p\cap a_q$, then
 $p$ and $q$ are compatible in $\PP$. This shows 
 that $\PP$ satisfies the c.c.c.

 Let $\dot{\mathbb G}$ be a $\PP$-name for a generic filter in $\PP$.  For  $\xi<\kappa$
 let $\dot A_\xi$ be a $\PP$-name for a subset of $\N$
 such that $\PP$ forces that $\dot A_\xi=\bigcup\{A_p(\xi): p\in  \dot{\mathbb G}\}$.
 Standard density arguments imply that $\PP$ forces that $\dot A_\xi$ is an infinite subset of $\N$
 and that 
 $$\A_{\dot{\mathbb G}}=\{\dot A_\xi: \xi<\kappa\}$$
 is an $\ad$-family. We will prove that $\A_{\dot{\mathbb G}}$ is the required family.

  Let $\{\dot \rho_\alpha: \alpha<\omega_1\}$ be $\PP$-names for elements of an uncountable subset of
   $\PP_{\A_{\dot{\mathbb G}}}$.
 By Lemma \ref{thinning} and by passing to a further uncountable subset we may assume that there are conditions
 $p_\alpha\in \PP$ for $\alpha<\omega_1$ and  $k, l, m\in \N$ and disjoint $a_\alpha\in [\omega_1]^k$, $b_\alpha\in [\omega_1]^l$
 and $E, F\subseteq m$ such that
 $$p_\alpha\forces \dot\rho_\alpha=\Big((\bigcup_{\xi\in \check a_\alpha}\dot A_\xi\setminus\check m)\cup\check E,
 (\bigcup_{\xi\in \check b_\alpha}\dot A_\xi\setminus\check m)\cup\check F\Big).\leqno (*)$$
 
 By  standard density arguments and passing to an uncountable subset we may assume that  for each $\alpha<\omega_1$ we have
 $n=n_{p_\alpha}\geq m$ and $a_\alpha, b_\alpha\subseteq a_{p_\alpha}$.  Moreover by another thinning out
 we may assume that there are bijections $\phi_{\beta, \alpha}: a_{p_\alpha}\rightarrow a_{p_\beta}$ such that for every
 $\alpha<\beta<\omega_1$  we have
 \begin{itemize}
 \item $\phi(\xi)=\xi$ for $\xi\in a_{p_\alpha}\cap a_{p_\beta}$,
 \item $A_{p_\alpha}(\xi)=A_{p_\beta}(\phi(\xi))$ for  $\xi\in a_{p_\alpha}$,
\item $\xi \in a_\alpha$ if and only if  $\phi(\xi)\in a_\beta$ for $\xi\in a_{p_\alpha}$,
\item $\xi\in b_\alpha$ if and only if  $\phi(\xi)\in b_\beta$ for $\xi\in a_{p_\alpha}$.
 \end{itemize}
 Using the same argument as in the proof of the c.c.c. of $\PP$ we see that any such conditions $p_\alpha$ and $p_\beta$
 are compatible, so the fact that $\dot a_\alpha\cup \dot b_\alpha$ must be forced to be disjoint
  with $\dot a_\beta\cup \dot b_\beta$
 implies that actually $a_\alpha\cup b_\alpha\subseteq a_{p_\alpha}\setminus a_{p_\beta}$ and
  $a_\beta\cup b_\beta\subseteq a_{p_\beta}\setminus a_{p_\alpha}$.
 
Pick any two $\alpha<\beta<\omega_1$. Put $p=p_\alpha$ and $q=p_\beta$.  First we will construct $r\leq p, q$
such that $r$ forces that $\dot\rho_\alpha$ and $\dot\rho_\beta$ 
are compatible in $\PP_{\A_{\dot G}}$. For this it is enough to
put $a_r=a_p\cup a_q$ and $n_r=n$  and 
$A_r(\xi)=A_p(\xi)$ for $\xi\in a_p$ and $A_r(\xi)=A_q(\xi)$ for $\xi\in a_q$.  For $\xi\in a_\alpha$ and $\eta\in b_\beta$ we have
$$(A_r(\xi)\cap A_r(\eta))\setminus m=(A_q(\phi(\xi))\cap A_q(\eta))\setminus m=\emptyset$$
since $\phi(\xi)\in a_\beta$ and $q$ forces that $\dot\rho_\beta$ is a condition of
$\PP_{\A_{\dot{\mathbb G}}}$.
So by (7) we have that $r$ forces that $(A_r(\xi)\cap A_r(\eta))\setminus m=\emptyset$ and so
by the choice of $\xi, \eta$  we have  that $r$ forces $\rho_\alpha$ to be compatible with $\rho_\beta$.

Now we will construct $s\leq p, q$
such that $s$ forces that $\rho_\alpha$ and $\rho_\beta$ are incompatible in $\PP_{\A_{\dot G}}$. 
Note that the hypothesis of the proposition implies that $a_\alpha, a_\beta, b_\alpha, b_\beta\not=\emptyset$.
So pick $\xi_0\in a_\alpha$ and $\eta_0\in b_\beta$.
Now
put $a_r=a_p\cup a_q$ and $n_r=n+1$ and 
$$
A_r(\xi)=
  \begin{cases}
    A_p(\xi)= A_q(\xi)& \text{if $\xi\in a_p\cap a_q$}, \\
    A_p(\xi)\cup\{n\}& \text{if $\xi=\xi_{0}$},\\
    A_p(\xi)& \text{if $\xi\in a_p\setminus a_q$ and  $\xi\not={\xi_0}$},\\
    A_q(\xi)\cup\{n\}& \text{if $\xi=\eta_{0}$},\\
    A_q(\xi)& \text{if $\xi\in a_q\setminus a_p$ and  $\xi\not={\eta_0}$}.
  \end{cases}
 $$
 It should be clear that $r\in \PP$ and $r\leq p, q$ moreover $r$ forces that
 $$\check n\in \dot (A(\xi_0)\setminus\check m)\cap (\dot A(\eta_0)\setminus\check m)$$ so
 $r$ forces that $\dot\rho_\alpha$ is incompatible with $\dot\rho_\beta$.
 
 For the maximality of $\A_{\dot{\mathbb G}}$ suppose that
 there is a $\PP$-name $\dot A$ for an infinite subset of $\N$ and
 $p\in \PP$ which forces that $\dot A$ is almost disjoint from all
 $\dot A_\xi$ for $\xi<\kappa$. It follows that there is
 a name $\dot X$ for an uncountable subset of $\omega_1$
 and $k\in \N$ and $q\leq p$ such that
 $$q\forces \dot A_\alpha\cap \dot A\subseteq \check k \ \hbox{for all} \ \alpha\in \dot X.\leqno (*)$$
 Let $p_\xi\forces \check\alpha_\xi\in \dot X$ for $\alpha_\xi\in a_{p_\xi}$ and
 $\alpha_\xi\not=\alpha_\eta$ for $\xi\not=\eta$. By passing to an uncountable
 subset of $\{p_\xi:\xi<\omega_1\}$ we may assume that $n_{p_\xi}=n_{p_\eta}=n\in \N$
 for all $\xi<\eta<\omega_1$ and 
 $\{a_{p_\xi}: \xi<\omega_1\}$ forms a $\Delta$-system with root $\Delta$ and
 $A_{p_\xi}(\alpha)=A_{p_\eta}(\alpha)$ for every $\alpha\in\Delta$ and $\xi<\eta<\omega_1$.
 Let $q=(n, \Delta, (A_{p_0}(\alpha): \alpha\in \Delta))$. There is $r\leq q$
 which decides an element $m$ of $\dot A$ above $n$. Now it is easy to see
 that there is $\xi<\omega_1$ such that $(a_{p_\xi}\setminus \Delta)\cap a_r$
 and one can construct $s\leq r, p_\xi$ such that $s$ forces
 that $m\in \dot A_{\alpha_\xi}$ which contradicts (*).

The argument that $\PP$ forces that $\mathfrak c$ is equal to $\kappa$ is standard (see e.g. \cite{jech}).
\end{proof}

\begin{proposition}[{\sf CH}]\label{antiramsey-ch}  There is an antiramsey $\ad$-family
 $\A$ which is a maximal $\ad$-family.
\end{proposition}
\begin{proof} 
Use {\sf CH} to construct  two enumerations.
First enumeration is $(\PP_\xi: \xi<\omega_1)$ of all countably infinite  families of elements
$p$ of the form $p=(A_p, B_p)$, where $A_p, B_p\subseteq\N$  and $A_p\cap B_p=\emptyset$.
The second one is $\{X_\xi:\xi<\omega\}$ of all infinite subsets of $\N$ such that $X_0=\N$.
We construct an antiramsey  $\ad$-family $\mathcal C=\{C_\alpha:\alpha<\omega_1\}$ by induction on
$\alpha<\omega_1$.   Let $\{C_n: n\in \omega\}$
be any family of pairwise disjoint infinite subsets of $\N$.  Having constructed 
$\mathcal C_\alpha= \{C_\beta: \beta<\alpha\}$ for $\omega\leq\alpha<\omega_1$  we will describe  the construction
of $C_\alpha$. Let $\I_\alpha$ be the
ideal generated by $\mathcal C_\alpha$ and finite subsets of $\N$.
We will consider a partial order $\Q_\alpha$ consisting of conditions
of the form $q=(n_q, s_q, I_q)$, where
\begin{itemize}
\item $n_q\in \N$
\item $s_q\in 2^{<\omega}$
\item $n_q=|s_q|$
\item $I_q\in\I_\alpha$ 
\item $s^{-1}[\{1\}]\cap I_q=\emptyset$.
\end{itemize}
and $q\leq r$ if $s_q\supseteq s_r$ and $I_q \supseteq I_r$ and $n_q\geq n_r$.
$$C_\alpha=\bigcup\{ s^{-1}[\{1\}]: s\in \mathbb G_\alpha\},$$
where $\mathbb G_\alpha$ is a filter in $\Q_\alpha$ which meets certain countable family of
dense subsets of $\Q_\alpha$. Below we will specify these dense sets.
First consider sets
\begin{itemize}
\item $D^1_{\xi, i}=\{q\in \Q_\alpha: \exists j>i\ q(j)=1,\  \&\  j\in X_\xi\}$ for $i\in \N$ and $\xi<\omega_1$
such that $X_\xi\not\in \I_\alpha$.
\item $D^2_\beta=\{q\in \Q_\alpha: C_\beta\subseteq^*I_q\}$ for $\beta<\alpha$.
\end{itemize}
It is clear that since $\mathcal C_\alpha$ is almost disjoint, infinite and consiting of infinite sets,
the above sets are dense in $\Q_\alpha$. 
It is clear that if $\mathbb G$ meets each $D^1_{X_0, i}$ and $D^2_\beta$ 
for all $i\in \N$ and all $\beta<\alpha$, then $C_\alpha$ is infinite
and almost disjoint from all $C_\beta$ for $\beta<\alpha$.

For $m\in \N$ and $\xi<\omega_1$ such that $\PP_\xi$
 is a family of essentially distinct elements of $\PP_{\mathcal C_\alpha}$
define further sets to be proved to be dense:
\begin{itemize}
\item $D^3_{\xi, m}=\{q\in  \Q_\alpha: \exists j>m\  \exists p\in \PP_\xi 
\  j\in A_p\cap s_q^{-1}[\{1\}]\}$.
\item $D^4_{\xi, m}=\{q\in \Q_\alpha: \exists j>m\  \exists p\in \PP_\xi
\  j\in B_p\cap s_q^{-1}[\{1\}]\}$.
\end{itemize}
Let us check that the above sets are dense in $\Q_\alpha$
for $m\in \N$ and $\xi<\omega_1$ as above.  
Fix $r\in \Q_\alpha$.
By the density of $D^1_m$ we may assume that $n_r\geq m$. 
As $\PP_\xi$ is infinite and consists of essentially distinct conditions of $\PP_{\mathcal C_\alpha}$, there is  $p\in \PP_\xi$ such
that $A_p\not\subseteq^* I_r$.  Let $j\in A_p\setminus (I_r\cup n_r)$. 
Define $q\leq r$ by putting 
$$s_q=s_r\cup(0|[n_r, j))\cup\{\langle j, 1\rangle\}, \ n_q=j+1, \ I_q=I_r.$$
It is clear that $q\in D^3_{\xi, m}$. An analogous argument works for
$D^4_{\xi, m}$.

Now for $m\in \N$, $p'\in \PP_{\mathcal C_\alpha}$ and $\xi<\omega_1$ such that
 $\PP_\xi\subseteq \PP_{\mathcal C_\alpha}$ define auxiliary sets
\begin{itemize}
\item $E_{\xi, p',  m}^1=\{q\in \Q_\alpha:  
\forall p\in \PP_\xi \  p\parallel p'\ \Rightarrow\ 
  s_q^{-1}[\{1\}]\setminus m\not\subseteq A_p\}$
  
  \item $E_{\xi, p',  m}^2=\{q\in \Q_\alpha: 
\forall p\in \PP_\xi \ p\parallel p'\ \Rightarrow\ 
  s_q^{-1}[\{1\}]\setminus m\not\subseteq B_p\}$

  \item $E_{\xi, p',  m}^3=\{q\in \Q_\alpha:  
  \exists p\in \PP_\xi,  \ p\parallel p'\  \& \
B_p\setminus m\subseteq I_q\}$.

\item $E_{\xi, p',  m}^4=\{q\in \Q_\alpha:  
  \exists p\in \PP_\xi,  \ p\parallel p'\  \& \
A_p\setminus m\subseteq I_q\}$.
\end{itemize}
Finally define the last batch of sets to be proven dense:

\begin{itemize}
\item $D^5_{\xi, p',  m}=E_{\xi, p',  m}^1\cup E_{\xi, p',  m}^3$.
\item $D^6_{\xi, p',  m}=E_{\xi, p',  m}^2\cup E_{\xi, p',  m}^4$.
\end{itemize}

We use this opportunity to note that 
\begin{enumerate}
\item[(*)] if there is $q\in \mathbb G_\alpha\cap E_{\xi, p',  m}^3$,
then there is $p\in \PP_\xi$ which 
is compatible with $(A_{p'}\cup (C_\alpha\setminus m), B_{p'})$ as long as 
 $(C_\alpha\setminus m)\cap B_{p'}=\emptyset$,
\end{enumerate} since
$C_\alpha\cap (B_p\setminus m)\subseteq (\N\setminus I_q)\cap I_q=\emptyset$.
Likewise  
\begin{enumerate}
\item[(**)] if there is $q\in \mathbb G_\alpha\cap E_{\xi, p',  m}^4$, 
then there is $p\in \PP_\xi$ which 
is compatible with $(A_{p'}, B_{p'}\cup (C_\alpha\setminus m))$ as long as
 $(C_\alpha\setminus m)\cap A_{p'}=\emptyset$.
\end{enumerate}

Let us prove that $D^5_{\xi, p',  m}$ is dense if $\PP_\xi$ is a collection
of essentially distinct elements of $\PP_{\mathcal C_\alpha}$, $p'\in \PP_{\mathcal C_\alpha}$ and $m\in \N$.
Consider $r\in \Q_\alpha$
such that $n_r\geq m$.  If
for all  $p\in \PP_\xi$ such that $p$ and $p'$ are compatible we have  $s_r^{-1}[\{1\}]\setminus m\not\subseteq A_p$
then $r\in E^1_{\xi, p', m}\subseteq D^5_{\xi, p', m}$. Otherwise find
$p\in \PP_\xi$ compatible with $p'$ such that   $s_r^{-1}[\{1\}]\setminus m\subseteq A_p$.
Consider 
$$s_q=(n_r, s_r, I_r\cup B_p\setminus m).$$ It is a condition
of $\Q_\alpha$ because $A_p\cap B_p=\emptyset$.
It is clear that $q\in  E^3_{\xi, p', m}\subseteq D^5_{\xi, p', m}$ 
and $q\leq r$ which completes the proof of the density
of $D^5_{\xi, p', m}$. An analogous argument works for  $D^6_{\xi, p', m}$.

Now we specify $\mathbb G_\alpha$ as a filter in $\Q_\alpha$ which meets
$D^1_{\xi, i}$s for all $i\in \N$ and $\xi<\alpha$ such that $X_\xi\not\in \I_\alpha$, $D^2_\beta$s for all $\beta<\alpha$,
$D^3_{\xi, m}$s and $D^4_{\xi, m}$s for all $m\in \N$
and $\xi<\alpha$ such that $\PP_\xi\subseteq \PP_{\mathcal C_\alpha}$
is a collection of essentially distinct conditions,
$D^5_{\xi, p', m}$s and $D^6_{\xi, p', m}$s for all
$m\in \N$, all $p'\in \PP_{\mathcal C_\alpha}$ and 
all $\xi<\alpha$ such that $\PP_\xi\subseteq \PP_{\mathcal C_\alpha}$
is a collection of essentially distinct conditions.
This completes the inductive step 
of the construction of the $\ad$ family $\mathcal C=\{C_\alpha: \alpha<\omega_1\}$.

Now we prove that $\mathcal C$ is antiramsey.  Consider an
uncountable family $\PP\subseteq \PP_{\mathcal C}$ of essentially distinct conditions.
By passing to an uncountable subfamily we may assume that
there are $m\in \N$ and pairwise disjoint $F, G\subseteq m$ such that
each $p\in \PP$ is of the form 
$$p=([(\bigcup a_p)\setminus m]\cup F, [(\bigcup b_p)\setminus m)\cup G)$$
for some finite $a_p, b_p\subseteq \mathcal C$ such that $a_p\cap b_p=\emptyset$ and 
$(a_p\cup b_p)\cap(a_q\cup b_q)=\emptyset$ for any distinct $p, q\in \PP$.
There is an infinite and  countable $\overline\PP\subseteq \PP$
such that for every $p'\in \PP$ and every $n\in \N$ there is
$p\in \overline\PP$ such that $A_p\cap n=A_{p'}\cap m$ and
$B_p\cap m=B_{p'}\cap m$. There is $\xi<\omega_1$ such that $\overline\PP=\PP_\xi$
and there is $\xi<\eta<\omega_1$ such that $\PP_\xi\subseteq \PP_{\mathcal C_\eta}$.
Pick $p'\in \PP$ such that $a_{p'}\cap(\eta+1)=\emptyset=b_{p'}\cap(\eta+1)$.
We will show that there is $p\in \PP_\xi$ such that $p$ and $p'$ are
compatible and that there is $p\in \PP'$ such that $p$ and $p'$ are
incompatible. 

Let $a_{p'}\cup b_{p'}=\{\alpha_1, \dots, \alpha_k\}$
in the increasing order.  For finite disjoint $F', G'\subseteq \N$
and $0\leq l\leq k$ define as $p_0'(F', G')=(F\cup F', G\cup G')$
and 
$$p_l'(F', G')=
([\bigcup (a_{p'}\cap \{C_{\alpha_1}\dots,C_{\alpha_l}\})\setminus m]\cup F\cup F', 
[\bigcup (b_{p'}\cap \{C_{\alpha_1}\dots,C_{\alpha_l}\})\setminus m]\cup G\cup G')$$
for $2\leq l\leq k$.
By induction on $1\leq l\leq k$ we will prove that for every finite $F'\subseteq A_{p'}$
and finite $G'\subseteq B_{p'}$ there is 
 there is $p\in \PP_\xi$ compatible with $p_l(F', G')$ and stronger than $(F', G')$. For
 $l=k$ and $F'=G'=\emptyset$ this will give the desired $p\in \PP_\xi$ compatible
 with $p'$.
 
 For $l=0$ this follows from the choice of $\PP_\xi$ since $p$ is
 compatible with $(F', G')$ if $F'\subseteq A_{p'}$ and $G'\subseteq B_{p'}$.

 To prove it for $l>0$, fix $F'\subseteq A_{p'}, G'\subseteq B_{p'}$ and first assume that
 $\alpha_l\in a_{p'}$
 and consider $q\in G_{\alpha_l}\cap D^5_{\xi, p_{l-1}'(F', G'), m}$.
 By the inductive assumption there is $p\in \PP_\xi$
 compatible with $p_{l-1}'(F'\cup (s_q^{-1}[\{1\}]\setminus m), G')$ and stronger than 
 $(F'\cup (s^{-1}_q[\{1\}]\setminus m), G')$, so $(s^{-1}_q[\{1\}]\setminus m)\subseteq A_p$
 and $p$ is compatible with $p_{l-1}'(F', G')$. It follows that 
 $q\not \in E^1_{\xi, p_{l-1}'(F', G'), m}$ and hence
  $q \in E^3_{\xi, p_{l-1}'(F', G'), m}$. By (*) it follows that 
  $p_{l}'(F', G')=(A_{p_{l-1}'(F', G')}\cup(C_{\alpha_l}\setminus m), B_{p_{l-1}'(F', G')})$ 
  is compatible with $p$
  which completes the inductive step and allows to conclude
  the existence of the desired $p\in \PP_\xi$ compatible
 with $p'$. The case of $\alpha_l\in b_{p'}$ is analogous and uses (**) instead of (*)
  and $D^6_{\xi, p_{l-1}'(F', G'), m}$ instead of  $D^5_{\xi, p_{l-1}'(F', G'), m}$.

To obtain $p\in \PP_\xi$ incompatible with $p'$ pick any $\alpha\in a_{p'}$
and use the fact that $\mathbb G_\alpha$ intersects $D^4_{\xi, m}$ to find
$j>m$ in $C_{\alpha}\cap B_p$. This guarantees that $p$ and $p'$ are incompatible.
This completes the proof of the fact that $\mathcal C$ is antiramsey.

To see that $\mathcal C$ is a maximal $\ad$-family, it is enough to prove that
whenever $X\subseteq\N$ is infinite, then there is $\alpha<\omega_1$ such that
$X\cap C_\alpha$ is infinite.  If $X\in \I=\bigcup_{\alpha<\omega_1}\I_\alpha$, there are
finitely many generators $C_\alpha$ of $\I$ which cover $X$, so at least one of them must
have infinite intersection with $X$. If $X\not\in\I$, then consider $\xi<\omega_1$
such that $X=X_\xi$ and $\alpha>\xi$. Since $\mathbb G_\alpha$ intersects
$D^1_{i, \xi}$ for all $i\in \N$ we conclude that $C_\alpha\cap X_\xi$ is infinite as required.

\end{proof}

\begin{proposition}\label{antiramsey-cohen}  In any model of {\sf ZFC} obtained by adding a Cohen real
there is an antiramsey $\ad$-family $\A$ of cardinality $\mathfrak \omega_1$.
\end{proposition}
\begin{proof} 
Let $V$ be the ground model, $\Q$ be the Cohen forcing  i.e., $\{0,1\}^{<\omega}$ with the end-extension as the order, 
and let $\dot c\in \{0, 1\}$ be the name for a Cohen real over $V$.
Let $\B=\{B_\xi: \xi<\omega_1\}$ be a Luzin family (in the sense of Definition \ref{def-L})  in $V$.
We claim that $\A=\{A_\xi: \xi<\omega_1\}$ satisfies the proposition in $V[c]$, where $\Q$ forces that
$$\dot A_\xi= \check B_\xi\cap \dot c^{-1}[\{1\}]$$
for every $\xi\in \omega_1$. 

It is clear that $\A$ is an $\ad$-family since its is a refinement of $\B$ and
it is well-known that a Cohen real intersects any ground model set on an infinite set.
Work in $V$ and suppose that $\{\dot p_\xi: \xi<\mathfrak \omega_1\}$
are $\Q$-names for conditions of $\PP_\A$  such that
$A_{\dot p_\xi}\not=A_{\dot p_\eta}$ and $B_{\dot p_\xi}\not=B_{\dot p_\eta}$ for $\xi<\eta$.
Let $p\in \Q$.

Deciding and using Lemma \ref{thinning} and using the countability of $\Q$ 
 we may assume that there are $a_\xi, b_\xi\in [\omega_1]^{<\omega}$
and $m\in \N$ and $E, F\subseteq m$ and a condition $q\in \Q$ such that $q\leq p$ forces that
$$\dot p_\xi=((\bigcup \{\dot A_\alpha: \alpha\in \check a_\xi\}\setminus \check m)\cup \check F, 
(\bigcup\{\dot A_\alpha: \alpha\in \check b_\xi\}\setminus \check m)\cup \check E)$$
and $E\cap F=\emptyset$, $a_\xi\cap b_\xi=\emptyset=(a_\xi\cup b_\xi)\cap (a_\eta\cup b_\eta)$
for any $\xi<\eta<\omega_1$. Let $k=|q|$. We may assume that $k\geq m$.
Find uncountable $\Gamma\subseteq \omega_1$ such that 
$$[(\bigcup \{B_\alpha: \alpha\in a_\xi\}\setminus m)\cup F]\cap k=
[(\bigcup \{B_\alpha: \alpha\in a_\eta\}\setminus m)\cup F]\cap k\leqno(1)$$
$$[(\bigcup \{B_\alpha: \alpha\in b_\xi\}\setminus m)\cup F]\cap k=
[(\bigcup \{B_\alpha: \alpha\in b_\eta\}\setminus m)\cup F]\cap k\leqno(2)$$
for every distinct $\xi, \eta\in \Gamma$. Of course it follows from 
the above that $q$ forces  that
$$[(\bigcup \{\dot A_\alpha: \alpha\in \check a_\xi\}\setminus \check m)\cup \check F]\cap \check k=
[(\bigcup \{\dot A_\alpha: \alpha\in \check a_\eta\}\setminus \check m)\cup \check F]\cap \check k\leqno(1)$$
$$[(\bigcup \{\dot A_\alpha: \alpha\in \check b_\xi\}\setminus \check m)\cup \check F]\cap \check k=
[(\bigcup \{\dot A_\alpha: \alpha\in \check b_\eta\}\setminus \check m)\cup \check F]\cap \check k\leqno(2)$$
for every distinct $\xi, \eta\in \Gamma$.

Now aim at finding $r\leq q$ and distinct $\xi, \eta\in \Gamma$ such that
$r$ forces that $\dot p_\xi$ and $\dot p_\eta$ are compatible.
Choose any distinct $\xi, \eta\in \Gamma$. Find $k'\geq k$ such that
$$\bigcup \{B_\alpha: \alpha\in a_\xi\}\cap \bigcup \{B_\alpha: \alpha\in b_\eta\}\setminus k'=\emptyset,$$
$$\bigcup \{B_\alpha: \alpha\in a_\eta\}\cap \bigcup \{B_\alpha: \alpha\in b_\xi\}\setminus k'=\emptyset.$$
The existence of such a $k'\in \N$ follows from the fact
that $(a_\xi\cup b_\xi)\cap(a_\eta\cup b_\eta)=\emptyset$ and 
$\B$ is an $\ad$-family. Extend $q$ to $r\in \{0,1\}^{k'}$
by all $0$s. This implies that $r$ forces that 
$$\big(\bigcup \{\dot A_\alpha: \alpha\in \check a_\xi\}\cap \bigcup
 \{\dot A_\alpha: \alpha\in \check b_\eta\}\big)\setminus \check k=\emptyset,$$
$$\big(\bigcup \{\dot A_\alpha: \alpha\in \check \check a_\eta\}\cap
 \bigcup \{\dot A_\alpha: \alpha\in \check b_\xi\}\big)\setminus \check k=\emptyset.$$
and so it forces  that $\dot p_\xi$ and $\dot p_\eta$ are compatible.

Now aim at finding $r\leq q$ and distinct $\xi, \eta\in \Gamma$ such that
$r$ forces that $\dot p_\xi$ and $\dot p_\eta$ are not compatible.
Pick $\xi\in \Gamma$ such that $\Gamma\cap\xi$ is infinite. Pick $\alpha'\in a_\xi$.
By the Luzin property of $\B$ there is $\eta\in \Gamma\cap \xi$ and $\beta'\in b_\eta$ such that
$$l=max(B_{\alpha'}\cap B_{\beta'})>\max(m, k).$$
 Extend $q$ to $r\in \{0,1\}^{l+1}$ so that $r(l)=1$.
 This implies that $r$ forces that 
 $$\check l\in 
(\bigcup \{B_\alpha: \alpha\in \check a_\xi\}\setminus\check m)\cap
 (\bigcup \{B_\alpha: \alpha\in \check b_\eta\}\setminus \check m).$$
and so that it forces that $\dot p_\xi$ and $\dot p_\eta$ are not compatible.

\end{proof}

\section{A graph induced by an $\ad$-family}

It is clear that given an $\ad$-family $\A$ the properties of $\PP_\A$ 
considered in the previous section can be expressed in the language of
the compatibility graph of $\PP_\A$
We consider a graph $\G(\A)$ (Definition \ref{def-graph}) closely related to 
the compatibility graph of $\PP_\A$ whose vertices are
pairs  of disjoint finite subsets of $\A$. We show that
this graph can also express these properties of $\PP_\A$ (Proposition \ref{graph}) and
so can be considered an alternative tool to characterize
the phenomena in the induced Banach spaces considered in the following sections.

\begin{definition} Let $A, B, C, D$ be subsets of $\N$, we define\footnote{Note that
$(A,B)\Join(C, D)\not=\emptyset$ is equivalent to
$\big((A\setminus B)\cup 
(C\setminus D)\big)
\cap 
\big((B\setminus  A)\cup 
( D\setminus  C)\big)\not=\emptyset$.} 
$$(A, B)\Join(C, D)=\big((A\setminus B)\cap (D\setminus C)\big)\cup 
\big((B\setminus A)\cap (C\setminus D)\big)$$
\end{definition}

\begin{definition}\label{def-graph} Suppose that $\A$ is an $\ad$-family.
We define a graph $\G(\A)$ whose vertices are elements of
$$\mathcal V(\A)=\{(a, b): a, b\in [\A]^{<\omega}: a\cap b=\emptyset\}.$$
and $((a,b), (a, b'))$ is an edge of $\G(\A)$ if and only if
$$(\bigcup a,\bigcup b)\Join(\bigcup a', \bigcup b')\not=\emptyset.$$
Here the union of an empty family is understood to be empty.
\end{definition}

In this section we show that the properties of the partial order
$\PP_\A$ considered in the previous section can be completely and naturally expressed in terms
of cliques and independent sets of $\G(\A)$. This provides 
another combinatorial manifestation of the properties of Banach spaces considered in the following sections.

\begin{lemma}\label{join} Let $A, B, C, D, F$ be subsets of $\N$ such that
\begin{itemize}
\item $A\cap B, C\cap D\subseteq F$,
\item $A\cap F=C\cap F$, $B\cap F= D\cap F$.
\end{itemize}
Then
\begin{enumerate}
\item $(A,B)\Join(C, D)\cap F=\emptyset$,
\item $\big((A,B)\Join(C, D)\big)\setminus F=\big((A\cap D)\cup (B\cap C)\big)\setminus F$.
\end{enumerate}
\end{lemma}
\begin{proof}
For (1) note that by the hypothesis $A\cap F=C\cap F$, $B\cap F= D\cap F$
we have $(A,B)\Join(C, D)\cap F=(A,B)\Join(A, B)\cap F$
and $(A,B)\Join(A, B)=\emptyset$.

For (2) by (1) we have $(A,B)\Join(C, D)=\big((A,B)\Join(C, D)\big)\cap (\N\setminus F)$.
Now the hypothesis  $A\cap B, C\cap D\subseteq F$ implies 
that $\big((A,B)\Join(C, D)\big)\cap (\N\setminus F)=\big((A\cap D)\cup (B\cap C)\big)\big)\cap
(\N\setminus F)$.
\end{proof}

\begin{definition} Suppose that $\A$ is an $\ad$ family.  We define
\begin{itemize}
\item $\triangleleft:\mathcal V(\A)\rightarrow \PP_\A.$
\item $\triangleright: \PP_A\rightarrow \mathcal V(\A)$
\end{itemize}
by
$$\triangleleft(a, b)=(\bigcup a\setminus \bigcup b, \bigcup b\setminus\bigcup a)$$
$$\triangleright(p)=(\{A\in \A: A\subseteq^*A_p\}, \{B\in \A: B\subseteq^*B_p\})$$
\end{definition}

\begin{lemma}\label{surjection} Suppose that $\A$ is an $\ad$ family. Then
\begin{itemize}
\item $\triangleright$ is surjective with countable fibers, $\triangleleft$ is injective,
\item $\triangleright(\triangleleft(a, b))=(a,b)$ for any $(a, b)\in \mathcal V(\A)$.

\end{itemize}
\end{lemma}
\begin{proof}
Countable fibers of $\triangleright$ is the consequence of the definition of $\subseteq^*$.
The rest of the first part of the lemma follows from the second part which is clear from the definitions of
 $\triangleright$ and $\triangleleft$
because $A\subseteq^*\bigcup a \setminus \bigcup b$ as well as $B\subseteq^*\bigcup b \setminus \bigcup a$  if and only if 
$A\in a, B\in b$.
\end{proof}

\begin{lemma} \label{sigmaP}Suppose that $\A$ is an $\ad$-family. Then
$\PP_\A=\bigcup\PP_n$, where $\triangleright|\PP_n$ is injective and for every $n\in \N$ and  $p, q\in \PP_n$ we  have
that
$$p\parallel q\ \hbox{if and only if}\  (\triangleright(p),\triangleright(q))\in \G(\A).$$
%Moreover for every $(a,b), (a', b')\in \mathcal V(\A)$ we have
%$$\triangleleft(a,b)\parallel \triangleleft(a',b')\ \hbox{if and only if}\  ((a,b), (a',b'))\in \G(\A).$$

\end{lemma}
\begin{proof} 
Given $p\in \PP_\A$ let $\triangleright(p)=(a_p, b_p)$ and let
$k\in \N$ and $F, F', G, G'\in [\N]^{<\omega}$ be such that
\begin{itemize}
\item $[(\bigcup a_p)\setminus k]\cap[(\bigcup b_p)\setminus k]=\emptyset$,
\item $A_p\setminus k=(\bigcup a_p)\setminus k$,
\item  $B_p\setminus k=(\bigcup b_p)\setminus k$,
\item $(\bigcup a_p)\cap k=F$,
\item $(\bigcup b_p)\cap k=G$,
\item $A_p\cap k=F'$,
\item $B_p\cap k=G'$.
\end{itemize}
We claim that the injectivity and the  equivalence from the lemma holds for
$p, q$ for which such quintuples $(k, F, F', G, G')$ are the same.  For such $p, q$ Lemma \ref{join} applies
to $(\bigcup a_p, \bigcup b_p)$ and  $(\bigcup a_q, \bigcup b_q)$.
So $((a_p, b_p), (a_q, b_q))\in \G(\A)$ if and only if $\triangleleft(a_p, b_p)\bowtie\triangleleft(a_q, b_q)=\emptyset$
if and only if 
$((\bigcup a_p)\setminus k, (\bigcup b_p)\setminus k)\bowtie
 ((\bigcup a_q)\setminus k, (\bigcup b_q)\setminus k)=\emptyset$.
 For such $p, q$ Lemma \ref{join} applies to
$(A_p, B_p)$ and $(A_q, B_q)$ as well.
So $p\parallel q$ if and only if $(A_p\setminus k , B_p\setminus k)\bowtie(A_q\setminus k, B_q\setminus k)=\emptyset$.
Since  $A_p\setminus k=(\bigcup a_p)\setminus k$,
$B_p\setminus k=(\bigcup b_p)\setminus k$ and $A_q\setminus k=(\bigcup a_q)\setminus k$,
$B_q\setminus k=(\bigcup b_q)\setminus k$,
we can conclude the equivalence of the lemma. 

Also if $p, q$ have the same  quintuples $(k, F, F', G, G')$ and 
$\triangleright(p)=(a_p, b_p)=(a_q, b_q)=\triangleright(q)$, then
$A_p\cap k=F'=A_q\cap k$, $B_p\cap k=G'=B_q\cap k$  and
$A_p\setminus k= (\bigcup a_p)\setminus k=(\bigcup a_q)\setminus k=A_q\setminus k$,
and $B_p\setminus k= (\bigcup b_p)\setminus k=(\bigcup b_q)\setminus k=B_q\setminus k$,
and so $p=q$ which allows us to conclude the injectivity.
\end{proof}

\begin{proposition}\label{graph} Suppose that $\A$ is an $\ad$-family and $\kappa$
is a cardinal of uncountable cofinality.
\begin{enumerate}
\item $\PP_\A$ is $\sigma$-centered if and only if $\G(\A)$ is the union of countably many cliques.
\item $\PP_\A$ has precaliber $\kappa$ if and only if 
every subset of $\G(\A)$ of cardinality $\kappa$ contains a clique of cardinality $\kappa$
\item $\PP_\A$ is c.c.c. if and only if $\G(\A)$ does not contain an uncountable independent family.
\item  Every family in $\PP_\A$ of essentially distinct elements is the union
of countably many antichains if and only if
every family $\G\subseteq \G(\A)$ such that $a\not=a'$ and $b\not=b'$
for any distinct $(a,b), (a', b')\in \G$ is the union of countably many independent sets.
\end{enumerate}
\end{proposition}
\begin{proof}
Note that two elements $p,q\in \PP(\A)$ are essentially distinct if and only if $a\not=a'$ and $b\not=b'$,
where $p=(a,b)$ and $q=(a', b')$. 
The properties of  being the union of countably many cliques, having a clique of cardinality $\kappa$
in any subgraph of cardinality $\kappa$, not containing an uncountable independent set, being the union
of countably many independent sets pass from a graph to any of its subgraphs
and hold for a graph which is the union of countably many subgraphs each having the property.
This and Lemma \ref{sigmaP} imply the proposition.

\end{proof}

\section{Forcing $\PP_\A$ and the geometry of the unit 
sphere of the Banach space $\X_\A$}

In this section given an $\ad$-family $\A$ we characterize certain natural geometric properties
of the unit sphere of the Banach space $(\X_\A, \|\ \|_\infty)$ using  combinatorial (in fact, rather forcing theoretic)  properties
of the splitting partial order $\PP_\A$ (Definition \ref{def-P}).  These characterizations are Propositions \ref{sphere-ccc},
\ref{sphere-precaliber}, \ref{sphere-sigma} and refer to $\PP_\A$ satisfying the c.c.c., having precaliber $\kappa$
for an uncountable  cardinal $\kappa$ and being $\sigma$-centered. This leaves
out the geometric properties corresponding to $\A$ being antiramsey or $L$-family (Definitions \ref{def-antiramsey},
\ref{def-LL}). These  properties cannot be as global as the properties from  Propositions \ref{sphere-ccc},
\ref{sphere-precaliber}, \ref{sphere-sigma}
 for elementary geometric reasons
(spheres always contain of cardinality $\mathfrak c$ of small diameters). 
The results we obtain in this direction refer to any sequence which induces, in some sense,
an uncountable subset of $\PP_\A$ consisting of essentially distinct conditions 
(Propositions \ref{C_0-antiramsey},
\ref{C_0-luzin}).
However, when we pass to an appropriate subspace of $\X_\A$ in the following section, we will
be able to extract from these results a natural geometric condition.

\begin{lemma}\label{charactXA} Let $\A$ be an $\ad$-family. 
If $f\in \X_\A$, then for every $A\in \A$ there exists $\lim_{n\in A}f(n)$. For every $\varepsilon>0$
and every $f\in \X_\A$  the set 
$\{A: |\lim_{n\in A}f(n)|>\varepsilon\}$ is finite.
\end{lemma}
\begin{proof}
Given $A\in \A$ the subset of $\ell_\infty$ consisting of $f$s such  $\lim_{n\in A}f(n)$
exists is closed
in the supremum norm and $\X_\A$ is included in it because each generator of $\X_\A$ is in this set. 

If $\{A: |\lim_{n\in A}f(n)|>\varepsilon\}$ were infinite, $f$ could not be approximated 
in the supremum norm by linear combinations of the generators of $\X_\A$.

\end{proof}

\begin{definition}\label{fp}
Suppose that $p=(A_p, B_p)\in \PP_\A$,  then $f_p\in \X_\A$ 
is defined as
\[
  f_p(n) =
  \begin{cases}
    1 & \text{if $n\in A_p$} \\
     -1 & \text{if $n\in B_p$} \\
    0 & \text{$n\in \N\setminus(A_p\cup B_p)$.}
  \end{cases}
\]
\end{definition}

\begin{lemma}\label{cond-from-function} Suppose that $f\in \X_\A$ and $\varepsilon>0$. Then  there is 
a condition $p(f,\varepsilon)\in \PP_\A$ such that
\begin{enumerate}
\item $A_{p(f,\varepsilon)}\supseteq \{n: f(n)\geq \varepsilon\}$
\item $B_{p(f,\varepsilon)}\supseteq \{n: f(n)\leq -\varepsilon\}$
\end{enumerate}
\end{lemma}
\begin{proof}
Find $A_1, \dots A_n\in \A$ for some $n\in \N$ and finitely supported $f'\in c_0$
such that $\|g-f\|_\infty\leq\varepsilon/2$, where
$$g=\sum_{1\leq i\leq n}r_i1_{A_i} +f'$$
for some $r_i\in \R$ for $1\leq i\leq n$.  As $g$ assumes finitely many values, 
$A_{p(f,\varepsilon)}=\{n: g(n)\geq \varepsilon/2\}$ and  $B_{p(f,\varepsilon)}=\{n: g(n)\leq -\varepsilon/2\}$ yield
a condition $p(f, \varepsilon)\in \PP_\A$. Let us verify that it satisfies the lemma. If $f(n)\geq\varepsilon$ for $n\in \N$,
then $g(n)\geq\varepsilon/2$ and so $n \in A_{p(f,\varepsilon)}$. A similar argument works for $B_{p(f,\varepsilon)}$.

\end{proof}

\begin{lemma}\label{cond-and-function} Suppose that $\A$ is an $\ad$-family and $p, q\in \PP_\A$ and $f, g\in \X_\A$
and $0<\varepsilon\leq 1$. Then
\begin{enumerate}
\item $p$ is compatible with $q$ if and only if  $\|f_p-f_q\|_\infty\leq 1$.
\item $\|f_p-f_q\|_\infty>1$ if and only if $\|f_p-f_q\|_\infty=2$ (if and only if $p$ is incompatible with $q$).
\item If $p(f,\varepsilon)$ is compatible with $p(g, \varepsilon)$, then $\|f-g\|_\infty\leq 1+\varepsilon$
\end{enumerate}
\end{lemma}
\begin{proof}
The first two clauses follow from the definitions of $\PP_\A$, $f_p$ and $f_q$. 
For the third clause, suppose that $p(f,\varepsilon)$ is compatible with $p(g, \varepsilon)$.
So $\{n: f(n)\geq\varepsilon\}\cap \{n: g(n)\leq -\varepsilon\}=\emptyset$ and
$\{n: g(n)\geq\varepsilon\}\cap \{n: f(n)\leq -\varepsilon\}=\emptyset$. It follows that
for any $n\in \N$ we have either 
\begin{itemize}
\item $\{f(n), g(n)\}\in (-\varepsilon, \varepsilon)^2$
or  
\item $\{f(n), g(n)\}\in [-1, -\varepsilon]\times [-1, \varepsilon)$ or
\item $\{f(n), g(n)\}\in [\varepsilon, 1]\times (-\varepsilon, 1]$ or
\item $\{f(n), g(n)\}\in [-1, \varepsilon)\times [-1, -\varepsilon]$ or
\item $\{f(n), g(n)\}\in ( -\varepsilon, 1]\times[\varepsilon, 1]$.
\end{itemize}
 In all cases $|f(n)-g(n)|\leq 1+\varepsilon$ as required in (3).

\end{proof}

The following is a version of Theorem 2.6 of \cite{mer-ck}:

\begin{proposition}\label{sphere-ccc} Suppose that $\A$ is an  
$\ad$-family and $\kappa$ is an uncountable cardinal.
The following conditions are equivalent:
\begin{enumerate}
\item $\PP_\A$ admits an antichain of cardinality $\kappa$,
\item The unit sphere of $\X_\A$ admits a $2$-equilateral set of cardinality $\kappa$,
\item The unit sphere of $\X_\A$ admits an $r$-equilateral set of cardinality $\kappa$ for some $r>1$,
\item The unit sphere of $\X_\A$ admits an $(1+\varepsilon)$-separated set of cardinality $\kappa$ for some $\varepsilon>0$.
\end{enumerate}
\end{proposition}
\begin{proof}
Suppose that (1) holds and $\{p_\xi:\xi<\kappa\}$ is an antichain in $\PP_\A$. Then  
$\{f_{p_\xi}:\xi<\kappa\}$ is $2$-equilateral subset of the sphere of $\X_\A$ by Lemma \ref{cond-and-function}.
The implications from (2) to (3) and from (3) to (4) are obvious. So assume (4), fix $\varepsilon>0$
and  let $\{f_\xi: \xi<\kappa\}$ be a $(1+\varepsilon)$-separated set  in  in the unit sphere of  $\X_\A$. 
Consider $p(f_\xi, \varepsilon/2)$ from Lemma \ref{cond-from-function} for every $\xi<\kappa$. 
These conditions must form an antichain, because otherwise by Lemma \ref{cond-and-function} (3)
we would have $\|f_\xi-f_\eta\|_\infty\leq 1+ \varepsilon/2$.
\end{proof}

\begin{proposition}\label{sphere-precaliber} Suppose that $\A$ is an $\ad$-family and $\kappa$ is an uncountable cardinal.
The following conditions are equivalent:
\begin{enumerate} 
\item $\PP_\A$ has precaliber $\kappa$.
\item For every $\{f_\xi: \xi<\kappa\}\subseteq S_{\X_\A}$ and every $0<\varepsilon<1$
there is $\Gamma\subseteq\kappa$ of cardinality $\kappa$ such that
the diameter of $\{f_\xi: \xi\in\Gamma\}$ is not bigger than $(1+\varepsilon)$.
\item For every $\{f_\xi: \xi<\kappa\}\subseteq S_{\X_\A}$ there is $0<\varepsilon<1$ and
there is $\Gamma\subseteq\kappa$ of cardinality $\kappa$ such that
the diameter of $\{f_\xi: \xi\in\Gamma\}$ is not bigger than $(1+\varepsilon)$.
\end{enumerate}
\end{proposition}
\begin{proof} Assume (1). Let $\{f_\xi: \xi<\kappa\}\subseteq S_{\X_\A}$ and $0<\varepsilon<1$.
By the precaliber $\kappa$ there is $\Gamma\subseteq\kappa$ such that
$\{p(f_\xi, \varepsilon): \xi\in \Gamma\}$ is pairwise compatible.
So by  Lemma \ref{cond-from-function} (3) $\|f_\xi-f_\eta\|_\infty\leq 1+\varepsilon$ for all distinct
$\xi, \eta\in \Gamma$. The implication from (2) to (3) is obvious.

Now assume (3). Let $\{p_\xi: \xi<\kappa\}\subseteq \PP_\A$. 
By (3) there is $\Gamma\subseteq \kappa$
of cardinality $\kappa$ and $0<\varepsilon<1$ such that $\|f_{p_\xi}-f_{p_\eta}\|_\infty\leq 1+\varepsilon$
for any two distinct $\xi, \eta\in\Gamma$.  It follows from Lemma \ref{cond-and-function} (1), (2) that
$\{p_\xi: \xi<\kappa\}$ are pairwise compatible.

\end{proof}

\begin{proposition}\label{sphere-sigma} Suppose that $\A$ is an $\ad$-family.
The following conditions are equivalent:
\begin{enumerate} 
\item $\PP_\A$ is $\sigma$-centered,
\item For every $0<\varepsilon<1$ the sphere of $\X_\A$ is the union
of countably many sets $\X_n$ such that $\|x-y\|_\infty\leq 1+\varepsilon$ for all $x, y\in \X_n$ for all $n\in \N$.
\end{enumerate}
\end{proposition}
\begin{proof} Assume (1). Fix $0<\varepsilon<1$. Let $\PP_n$ for $n\in \N$ be centered families
such that $\bigcup_{n\in \N}\PP_n=\PP_\A$.  Define 
$$\X_n=\{ f\in S_{\X_\A}:  p(f, \varepsilon)\in \PP_n\}.$$
It follows that the sphere of $\X_\A$ is covered by the $\X_n$'s.  If $f, g\in \X_n$, then
$p(f, \varepsilon), p(g, \varepsilon)\in \PP_n$ and so $p(f, \varepsilon)$ and $p(g, \varepsilon)$ are compatible.
So by Lemma \ref{cond-and-function} (3), we conclude that $\|f-g\|_\infty\leq 1+\varepsilon$ as required for (2).

Now assume (2).  Define
$$\PP_n=\{p\in \PP_\A: f_p\in \X_n\}.$$
Suppose that $p, q\in \PP_n$. Since $\|f_p-f_q\|_\infty\leq 1+\varepsilon$, Lemma \ref{cond-and-function} (1) and (2) imply that
$p$ is compatible with $q$ which completes the proof of the lemma.
\end{proof}

\begin{definition}\label{cond-from-function2} Suppose that $f\in \X_\A$ and $0<\varepsilon<1/2$. Then we define
$A_{p[f,\varepsilon]}=A\setminus B$ and $B_{p[f,\varepsilon]}=B\setminus A$, where
\begin{enumerate}
\item $A=\bigcup\{n\in \N: \exists C\in \A \  \lim_{k\in C}f(k)\geq 1-\varepsilon \  \& \ n\in C\ \& f(n)>1-2\varepsilon\}$,
\item $B=\bigcup\{n\in \N: \exists C\in \A \  \lim_{k\in C}f(k)\leq -1+\varepsilon \  \&\  n\in C\ \& f(n)<-1+2\varepsilon\}$.
\end{enumerate}
\end{definition}

\begin{lemma} Suppose $f\in \X_\A$ and $0<\varepsilon<1/2$. 
Then $p[f,\varepsilon]=(A_{p[f,\varepsilon]}, B_{p[f,\varepsilon]}) \in \PP_\A$.
\end{lemma}
\begin{proof} Clearly the sets $A_{p[f,\varepsilon]}$ and $B_{p[f,\varepsilon]}$ are disjoint. So we need to
prove that these sets are almost equal to the union of some finite subfamily of $\A$. Let us focus on $A_{p[f,\varepsilon]}$,
the case of $B_{p[f,\varepsilon]}$ is analogous. It is enough to prove that the family of $C\in \A$ such that
$\lim_{n\in C}f(n)\geq 1-\varepsilon$ is finite.  This follows from Lemma \ref{charactXA}.
\end{proof}

\begin{remark} Note that given $f\in \X_\A$ the pair $(\{n\in \N: f(n)\geq \varepsilon\}, \{n\in \N: f(n)\leq -\varepsilon\}$
or the pair $(\{n\in \N: f(n)> \varepsilon\}, \{n\in \N: f(n)< -\varepsilon\}$ do not need to be the conditions of $\PP_\A$.
This is because $\varepsilon1_A+\Sigma_{n\in \N}[(-1)^n/n]1_{n}$ is in $\X_\A$.
\end{remark}

\begin{lemma}\label{two-conditions} Suppose $f\in \X_\A$ and $0<\varepsilon<1/3$.  Then
$A_{p[f,\varepsilon]}\subseteq A_{p(f,\varepsilon)}$ and $B_{p[f,\varepsilon]}\subseteq B_{p(f,\varepsilon)}$.
\end{lemma}

\begin{proposition}\label{C_0-antiramsey} Let $0<\varepsilon<1/2$.
Let $\A$ be an   $\ad$-family.    The following conditions are equivalent:
\begin{enumerate} 
\item $\A$ is an antiramsey  $\ad$-family.   
\item Whenever  $\{f_\xi:\xi<\omega_1\}\subseteq  S_{\X_\A}$ is such that 
$\{p[f_\xi, \varepsilon/4]: \xi<\omega_1\}$ consists of essentially distinct conditions,
then 
\begin{enumerate}
\item there are $\xi<\eta<\omega_1$ such that $\|f_\xi-f_\eta\|_\infty>2-\varepsilon$,
\item there are $\xi<\eta<\omega_1$ such that $\|f_\xi-f_\eta\|_\infty<1+\varepsilon$.
\end{enumerate}
\end{enumerate}
\end{proposition}
\begin{proof}  First let us prove that (1) implies (2).
Let $a_\xi, b_\xi\in [\A]^{<\omega}$, $m_\xi\in \N$, $F_\xi, G_\xi\in [\N]^{<\omega}$ be such that
$A_{p(f_\xi, \varepsilon/4)}=(\bigcup a_\xi\setminus m_\xi)\cup F_\xi$ and
$B_{p(f_\xi, \varepsilon/4)}=(\bigcup b_\xi\setminus m_\xi)\cup G_\xi$.
By passing to an uncountable subset we may assume that $m_\xi=m$, $G_\xi=G$ and $F_\xi=F$
for some  $m\in \N$, $F, G\in [\N]^{<\omega}$ and that $\{a_\xi: \xi<\omega_1\}$ forms
a $\Delta$-system with root $\Delta_1$ and   that $\{b_\xi: \xi<\omega_1\}$ forms
a $\Delta$-system with root $\Delta_2$.  Both of these $\Delta$-systems must be uncountable,
because otherwise we would have $a_\xi\subseteq \Delta_1$ or $b_\xi\subseteq \Delta_2$
for uncountably many $\xi<\omega_1$ and this would contradict, by Lemma \ref{two-conditions} the fact that
all $p[f_\xi, \varepsilon/4]$ are essentially distinct. It follows that we may assume that
$\{p(f_\xi, \varepsilon/4): \xi<\omega_1\}$ consists of essentially distinct conditions as well.

Let $\xi, \eta<\omega$ are such that $p[f_\xi, \varepsilon/4]$ and $p[f_\eta, \varepsilon/4]$ are incompatible. 
This means that there is $n\in A_{p[f_\xi, \varepsilon/4]}\cap B_{p[f_\eta, \varepsilon/4]}$ or
there is $n\in A_{p[f_\xi, \varepsilon/4]}\cap B_{p[f_\eta, \varepsilon/4]}$. By Definition \ref{cond-from-function2}
in the first case we have $f_\xi(n)>1-2\varepsilon/4$ and $f_\eta(n)<-1+2\varepsilon/4$ 
and in the second case $f_\xi(n)>-1+2\varepsilon/4$ and $f_\eta(n)>1-2\varepsilon/4$.
In any case $\|f_\xi-f_\eta\|_\infty>2-\varepsilon$.

Let $\xi, \eta<\omega$ be such that $p(f_\xi, \varepsilon/4)$ and $p(f_\eta, \varepsilon/4)$ are compatible. 
Then $\|f_\xi-f_\eta\|_\infty\leq 1+\varepsilon/4$ by Lemma \ref{cond-and-function} (3).

Now let us prove that (2) implies (1). Consider a family $\{p_\xi: \xi<\omega_1\}$ od essentially distinct
conditions of $\PP_\A$.  
Let $a_\xi, b_\xi\in [\A]^{<\omega}$, $m_\xi\in \N$, $F_\xi, G_\xi\in [\N]^{<\omega}$ be such that
$A_{p_\xi}=(\bigcup a_\xi\setminus m_\xi)\cup F_\xi$ and
$B_{p_\xi}=(\bigcup b_\xi\setminus m_\xi)\cup G_\xi$.
By passing to an uncountable subset we may assume that $m_\xi=m$, $G_\xi=G$ and $F_\xi=F$
for some  $m\in \N$, $F, G\in [\N]^{<\omega}$.
Let $p_\xi'=(\bigcup a_\xi\setminus m, \bigcup b_\xi\setminus m)$. Note that
$p_\xi'$ and $p_{\eta}'$ are compatible if and only if $p_\xi$ and $p_\eta$ are compatible
and that $p_\xi$'s are essentially distinct.
Also note that $p[f_{p_\xi'}, \varepsilon/4]=p_{\xi}'$ for every $\xi<\omega_1$. 
So (2) yields  $\xi<\eta<\omega_1$ such that $\|f_{p_\xi}-f_{p_\eta}\|_\infty>2-\varepsilon$,
which gives that $p_\xi$ and $p_\eta$ are incompatible by Lemma \ref{cond-and-function} (1), and also
yields  $\xi<\eta<\omega_1$ such that $\|f_{p_\xi}-f_{p_\eta}\|_\infty<1+\varepsilon$,
which gives that $p_\xi$ and $p_\eta$ are compatible by Lemma \ref{cond-and-function} (2). This completes the proof of (1).
\end{proof}

\begin{proposition}\label{C_0-luzin} Let $0<\varepsilon<1/2$. 
Let $\A$ be an $\ad$-family. The following conditions are equivalent:
\begin{enumerate}
\item $\A$ is an $L$-family.
\item Whenever  $\{f_\xi:\xi<\omega_1\}\subseteq  S_{\X_\A}$ is such that 
$\{p[f_\xi, \varepsilon/4]: \xi<\omega_1\}$ consists of essentially distinct conditions.
Then $\{f_\xi:\xi<\omega_1\}$ is a union of countably many $(2-\varepsilon)$-separated subfamilies.
\end{enumerate}
\end{proposition}
\begin{proof} First let us prove that (1) implies (2).
For $n\in \N$ let $\Gamma_n\subseteq\omega_1$ be such that $\{p[f_\xi, \varepsilon/4]: \xi\in \Gamma_n\}$ is
pairwise incompatible for each $n\in \N$ and $\bigcup_{n\in \N}\Gamma_n=\omega_1$.
Fixing $n\in \N$ and distinct $\xi, \eta\in \Gamma_n$ 
there is $k\in A_{p[f_\xi, \varepsilon/4]}\cap B_{p[f_\eta, \varepsilon/4]}$ or
there is $k\in A_{p[f_\xi, \varepsilon/4]}\cap B_{p[f_\eta, \varepsilon/4]}$. By Definition \ref{cond-from-function2}
in the first case we have $f_\xi(k)>1-2\varepsilon/4$ and $f_\eta(k)<-1+2\varepsilon/4$ 
and in the second case $f_\xi(k)>-1+2\varepsilon/4$ and $f_\eta(k)>1-2\varepsilon/4$.
In any case $\|f_\xi-f_\eta\|_\infty>2-\varepsilon$.

Now let us prove that (2) implies (1). Consider a family $\{p_\xi: \xi<\omega_1\}$ od essentially distinct
conditions of $\PP_\A$.  
Let $a_\xi, b_\xi\in [\A]^{<\omega}$, $m_\xi\in \N$, $F_\xi, G_\xi\in [\N]^{<\omega}$ be such that
$A_{p_\xi}=(\bigcup a_\xi\setminus m_\xi)\cup F_\xi$ and
$B_{p_\xi}=(\bigcup b_\xi\setminus m_\xi)\cup G_\xi$.
By passing to a part of a countable decomposition of $\omega_1$
 we may assume that $m_\xi=m$, $G_\xi=G$ and $F_\xi=F$
for some  $m\in \N$, $F, G\in [\N]^{<\omega}$.
Let $p_\xi'=(\bigcup a_\xi\setminus m, \bigcup b_\xi\setminus m)$. Note that
$p_\xi'$ and $p_{\eta}'$ are compatible if and only if $p_\xi$ and $p_\eta$ are compatible
and that $p_\xi$'s are essentially distinct.
Also note that $p[f_{p_\xi'}, \varepsilon/4]=p_{\xi}'$ for every $\xi<\omega_1$. 
So (2) yields  a decomposition $\omega_1=\bigcup_{n\in \N}\Gamma_n$ such that for each $n\in \N$ and 
distinct $\xi<\eta<\Gamma_n$ we have $\|f_{p_\xi}-f_{p_\eta}\|_\infty>2-\varepsilon$. This
which gives that $p_\xi$ and $p_\eta$ are incompatible by Lemma \ref{cond-and-function} (1).
This completes the proof of (1).

\end{proof}

\section{Subspaces of the Banach space $\X_\A$ for antiramsey and Luzin $\ad$-families}

In this section we pass to certain subspaces $\X_{\A, \phi}$ of $\X_\A$ in the case
of antiramsey and Luzin $\ad$-families $\A$ and exploit Propositions \ref{C_0-antiramsey}
and \ref{C_0-luzin} to obtain Banach spaces with interesting geometry of
the unit sphere (Proposition \ref{Y-antiramsey} and \ref{Y-luzin}). The point of
passing to these subspace is to free onself from the hypothesis on essentially distinct
elements of the sphere in Propositions \ref{C_0-antiramsey}
and \ref{C_0-luzin} and replace it with a condition which can be expressed
in the language of Banach spaces. Recall Definition \ref{def-antiramsey}
of an antiramsey $\ad$-family.

\begin{definition}\label{pairing}
 Let $\kappa$ be an uncountable cardinal and $\A$ be an $\ad$-family of cardinality at least $\kappa$. 
A function $\phi=(\phi_{1}, \phi_{-1}): \kappa\rightarrow \A$ is called 
a pairing of $\A$ if $\phi_1(\alpha)\not=\phi_{-1}(\alpha)$ for every $\alpha<\kappa$ and 
$\{\phi_{1}(\alpha), \phi_{-1}(\alpha)\}\cap\{\phi_{1}(\beta), \phi_{-1}(\beta)\}=\emptyset$ for any $\alpha<\beta<\kappa$.
\end{definition}

\begin{definition}\label{def-subspace}
Let  $\kappa$ be an uncountable cardinal and $\A$ be  an $\ad$-family of cardinality $\kappa$ 
and $\phi=(\phi_{1}, \phi_{-1}): \kappa\rightarrow \A$ be 
a pairing of $\A$. We define the Banach space $\X_{\A, \phi}$ as the linear span of
$$c_0\cup\{1_{\phi_{1}(\alpha)}-1_{\phi_{-1}(\alpha)}: \alpha<\kappa\}$$
in $\ell_\infty$ with the supremum norm. 
\end{definition}

\begin{proposition}\label{Y-antiramsey} Suppose that $\A$ is an antiramsey 
$\ad$-family and suppose that $\phi: \kappa\rightarrow \A$ is a pairing of $\A$. Let
 $\varepsilon\leq 1/3$ and 
 $\{f_\xi: \xi<\omega_1\}$ is a $(1-\varepsilon)$-separated subset of the sphere
$S_{\X_{\A, \phi}}$. Then  
\begin{enumerate}
\item there are $\xi<\eta<\omega_1$ such that
$\|f_\xi-f_\eta\|_\infty\leq 1+2\varepsilon$ and 
\item there are $\xi<\eta<\omega_1$ such that
$\|f_\xi-f_\eta\|_\infty\geq 2-8\varepsilon$. 
\end{enumerate}
\end{proposition}
\begin{proof} For every $\xi<\omega_1$ find $n_\xi, m_\xi\in \N$ and  rationals
$q_1^\xi, \dots q_{n_\xi}^\xi$ satisfying $0<|q_i^\xi|\leq 1$ for all $1\leq i\leq n$ and increasing 
$\alpha^\xi_1<\dots <\alpha^\xi_{n_\xi}<\omega_1$
 and rationally valued $g_\xi\in c_0$ whose support is included in $m_\xi$ such that
$\|f_\xi-h_\xi\|_\infty<\varepsilon$, and $\phi_{l}(\alpha_{\xi_i})\cap \phi_{l'}(\alpha_{\xi_j})\subseteq m_\xi$
for all  distinct pairs $\langle l, i\rangle$ and $\langle l', j\rangle$ for  $1\leq i, j\leq n$ and $l,l'\in \{-1,1\}$, where
$$h_\xi=g_\xi+\sum_{1\leq i\leq n}q_i(1_{\phi_{1}(\alpha_{\xi_i})\setminus m_\xi}-1_{\phi_{-1}(\alpha_{\xi_i})\setminus m_\xi}).$$
By passing to an uncountable subset of $\kappa$ we may assume that 
 there are $n, m\in \N$  and rationals
$q_1, \dots, q_{n}$ and finitely supported $g\in c_0$
such that $n_\xi=n$, $m_\xi=m$ and $g_\xi=g$ and $q_i=q_i^\xi$ for all  $1\leq i\leq n$.

Moreover, by the $\Delta$-system lemma,
 we may assume that there is $1\leq k<n$  and $\{\alpha_i: i<k\}\subseteq\omega_1$  such that $\alpha_i=\alpha^\xi_i$ for all $i<k$
 and $\{\alpha^\xi_k, \dots, \alpha^\xi_{n}\}\cap \{\alpha^\eta_k, \dots, \alpha^\eta_{n}\}=\emptyset$ for every $\xi<\eta<\omega_1$.
Indeed $k$ must be less than $n$ because otherwise $h_\xi=h_\eta$ and so $\|f_\xi-f_\eta\|_\infty<2\varepsilon$ which contradicts
$\|f_\xi-f_\eta\|_\infty>1-\varepsilon$ as $\varepsilon\leq 1/3$.  
Define 
$$A_{p_\xi}=\{n\in \N: (h_\xi-g)(n)> 0\}, \ \ B_{p_\xi}=\{n\in \N: (h_\xi-g)(n)<0\}.$$
It is easy to see that $p_\xi=(A_{p_\xi}, B_{p_\xi})\in \PP_\A$ because $h_\xi-g_\xi$ is constant equal to
non-zero $\pm q_i$  on each $\phi_{l}(\alpha^\xi_i)\setminus m$ for each $l\in \{-1,1\}$ and these sets are pairwise disjoint and are
almost equal to elements of the family $\A$.

Since $\PP_\A$ is assumed to satisfy the c.c.c,  there are $\xi<\eta<\omega_1$ such that
$p_{\xi}$ and $p_{\eta}$ are compatible.  This means that the supremum of the values
of $h_{\xi}-h_{\eta}=(h_{\xi}-g)-(h_{\eta}-g)$ is at most $\max_{k\leq i\leq n}\|q_i\|_\infty\leq 1$. 
In particular, $\|f_\xi-f_\eta\|_\infty\leq 1+2\varepsilon$ as required in 1).

Also the triangle inequality and the hypothesis that $\|f_\xi-f_\eta\|_\infty>1-\varepsilon$ for any $\xi<\eta<\omega_1$ imply that
$\|h_\xi- h_\eta\|_\infty>1-3\varepsilon$, so 
$\max_{k\leq i\leq n}\|q_i\|_\infty\geq 1-3\varepsilon$.
Let $1\leq i\leq n$ be such that
$$|q_i|\geq 1-3\varepsilon\leqno *)$$

Now for $\xi<\omega_1$ consider 
$$q_\xi=(\phi_1(\alpha^\xi_i)\setminus m_\xi,  \phi_{-1}(\alpha^\xi_i)\setminus m_\xi).$$
These are conditions of $\PP_\A$ by the choice of $m_\xi$. By Definition of \ref{pairing}
the $q_\xi$s are essentially different and so by the hypothesis on $\A$ there are $\xi<\eta<\omega_1$
such that $q_\xi$ and $q_\eta$ are incompatible. This means that there 
is $n\in n$ such that $|h_\xi- h_\eta)(n)|=|q_i-(-q_i)|=2|q_i|$ which gives
that $\|h_\xi-h_\eta\|_\infty\geq 2-6\varepsilon$ by *),
 and so $\|f_\xi-f_\eta\|_\infty\geq 2-8\varepsilon$  as required in 2).
\end{proof}

Recall Definition \ref{def-LL} of an $L$-family.

\begin{proposition}\label{Y-luzin} Suppose that $\A$ is an $\ad$-family
which is an $L$-family and $\kappa$ is an
uncountable cardinal.  Suppose that $\phi: \kappa\rightarrow \A$ is a pairing of $\A$. Let
 $\varepsilon\leq 1/3$ and 
 $\{f_\xi: \xi<\omega_1\}$ be a subset of the sphere $S_{\X_{\A, \phi}}$ such that
 $$\|x- f_\eta\|_\infty\geq 1-\varepsilon  \ \hbox{for any}\ x\in span(c_0\cup\{f_\xi: \xi<\eta\}),$$
 for every $\eta<\omega_1$. Then
 $\{f_\xi: \xi<\omega_1\}$ is the union of countably many $(2-5\varepsilon)$-separated sets.
\end{proposition}
\begin{proof}
 For every $\xi<\omega_1$ find $n_\xi, m_\xi\in \N$ and  rationals
$q_1^\xi, \dots q_{n_\xi}^\xi$ satisfying $0<|q_i^\xi|\leq 1$ 
for all $1\leq i\leq n$ and increasing $\alpha^\xi_1<\dots <\alpha^\xi_{n_\xi}<\kappa$
 and rationally valued $g_\xi\in c_0$ whose support is included in $m_\xi$ such that
$\|f_\xi-h_\xi\|_\infty<\varepsilon/2$, and $\phi_{l}(\alpha_{\xi_i})\cap \phi_{l'}(\alpha_{\xi_j})\subseteq m_\xi$
for all  distinct pairs $\langle l, i\rangle$ and $\langle l', j\rangle$ 
for  $1\leq i, j\leq n$ and $l,l'\in \{-1,1\}$, where
$$h_\xi=g_\xi+\sum_{1\leq i\leq n}q_i(1_{\phi_{1}(\alpha_{\xi_i})\setminus m_\xi}-1_{\phi_{-1}(\alpha_{\xi_i})\setminus m_\xi}).$$
For $k\in \N$ there are $\Gamma_k\subseteq \omega_1$ such that $\omega_1=\bigcup_{k\in \N}\Gamma_k$ and
for each $k\in \N$ and each $\xi\in \Gamma_k$
 there are $n, m\in \N$  and rationals
$q_1, \dots, q_{n}$ and finitely supported $g\in c_0$
such that $n_\xi=n$, $m_\xi=m$ and $g_\xi=g$ and $q_i=q_i^\xi$ for all  $1\leq i\leq n$.
Moreover by refining further the partition into $\Gamma_k$s and using
the Hewitt-Marczewski-Pondiczery theorem we may assume that if $\xi, \eta\in \Gamma_k$
and $\alpha=\alpha^\xi_i=\alpha^\eta_j$, then $i=j$. 
It is enough to show for each $k\in \N$ that $\{f_\xi: \xi\in\Gamma_k\}$ is 
the union of countably many $(2-\varepsilon)$-separated sets. So fix $k\in \N$. 

First we claim that for every distinct $\xi, \eta\in \Gamma_k$ there is $1\leq i\leq n$
such that $|q_i|\geq 1-2\varepsilon$ and  $\alpha^\xi_i\not=\alpha^\eta_i$.
 Indeed, otherwise
$$h_\xi-h_\eta=\sum_{i\in X}q_i(1_{\phi_{1}(\alpha_{\xi_i})\setminus m}-
1_{\phi_{-1}(\alpha_{\xi_i})\setminus m})-\sum_{i\in X}q_i(1_{\phi_{1}(\alpha_{\eta_i})
\setminus m}-1_{\phi_{-1}(\alpha_{\eta_i})\setminus m})$$
where $X=\{1\leq i\leq n: \alpha^\xi_i\not=\alpha^\eta_i\}$. Then
$\{\phi_j(\alpha^\xi_i): i\in X\}\cap \{\phi_j(\alpha^\eta_i): i\in X\}=\emptyset$ for any $j\in\{1, -1\}$.
Then the almost disjointness of $\A$  implies that there is finitely supported $h\in c_0$
such that $\|h_\xi-h_\eta-h\|_\infty\leq\max_{i\in X}\|q_i\|_\infty<1-2\varepsilon$ which means that
the distance from
$span(c_0\cup\{f_\xi\})$ to $f_\eta$ is less than $1-\varepsilon$ contradicting the hypothesis of the proposition.

Given $\xi\in \Gamma_k$ consider $p_\xi=(A_\xi, B_\xi)$, where
$$A_\xi=\bigcup\{\phi_1(\alpha_i^\xi)\setminus m: |q_i|\geq 1-2\varepsilon\}, \
\
B_\xi=\bigcup\{\phi_{-1}(\alpha_i^\xi)\setminus m: |q_i|\geq 1-2\varepsilon\}.$$
By the above claim for every $\xi, \eta\in \Gamma_k$ the conditions $p_\xi, p_\eta$ are
essentially distinct, so $\{p_\xi: \xi\in \Gamma_k\}$ is the union
of countably many antichains in $\PP_\A$ by the hypothesis of the proposition.
However, if $p_\xi$ and $p_\eta$ are incompatible there is
$n\in \N$ such that $n\in \phi_1(\alpha_i^\xi)\cap \phi_{-1}(\alpha_i^\eta)\setminus m$
or $n\in \phi_1(\alpha_i^\eta)\cap \phi_{-1}(\alpha_i^\xi)\setminus m$  and
$|q_i|\geq 1-2\varepsilon$. This means that $\|h_\xi-h_\eta\|_\infty\geq 2-4\varepsilon$
and consequently  $\|f_\xi-f_\eta\|_\infty\geq 2-5\varepsilon$ as required in the proposition.

\end{proof}

\section{Renormings of  Banach spaces $\X_\A$}

Two norms on the same vector space are said to be equivalent if the identity 
considered as a function between the metric spaces induced by the norms
is continuous in both directions.  For Banach spaces $(\X, \|\ \|_1)$ and $(\X, \|\ \|_2)$
this reduces to the existence of constants $c, C>0$
such that $c\|x\|_1\leq \|x\|_2\leq C\|x\|_1$ for all $x\in \X$.
Given a Banach space $(\X, \|\ \|_\X)$ we consider its equivalent renorming
$(\X, \|\ \|_T)$, where $T: \X\rightarrow \Y$ is a linear bounded operator for some Banach space $\Y$ and
$$\|x\|_T=\|x\|_\X+ \|T(x)\|_\Y$$
for every $x$ is $\X$.  The equivalence of the norms $\| \ \|_\X$ and $\|\ \|_T$
follows from the fact that $\|x\|_\X\leq\|\ x\|_T\leq (1+\|T\|)\|x\|_\X$ for every $x\in \X$.
In particular we consider  renormings of Banach spaces $\X_\A$ for
an $\ad$-family $\A$ with the norm of the form $\|\ \|_T$, where $T$ is injective and has separable range. 
More concretely we consider the operator $T:\X_\A\rightarrow \ell_2$ 
given by
$$T(f)=\Big({f(n)\over \sqrt{2^{n+1}}}\Big)_{n\in \N}$$
for any $f\in \X_\A$.
 So the new equivalent norm $\|\ \|_{T}$ denoted by $\|\ \|_\T$ on
$\X_\A$ is
$$\|f\|_{\T}=\|f\|_\infty+ \sqrt{\sum_{n\in \N} {f(n)^2\over 2^{n+1}}}.$$
For details see Section 2 of \cite{pk-kottman}.

We characterize
geometric properties of the unit spheres of $(\X_\A, \|\ \|_\T)$
by the combinatorial properties of the partial order $\PP_\A$ (Propositions \ref{sigma-T},
\ref{precaliber-T}, \ref{ccc-T}, \ref{sphereT-PA}).  In Proposition
\ref{oca-T} we translate the dichotomy \ref{oca} into the language of 
the Banach space $(\X_\A, \|\ \|_\T)$.

Recall from Section 2.1 that when we want to specify the norm $\|\ \|$
in a given Banach spaces $\X$, its unit sphere is denoted by $S_{\X, \|\ \|}$
First we need a few technical lemmas relating the distances in the spheres
$S_{\X, \|\ \|_\X}$ and $S_{\X, \| \ \|_T}$ for a given norm $\|\ \|_\X$ and
the norm $\|\ \|_T$ described above   and induced by a linear bounded
operator $T:\X\rightarrow \Y$, where $\X$ is considered with the norm $\|\ \|_\X$
and $\Y$ is a Banach space.

\begin{lemma}[Lemma 9 \cite{pk-kottman}]\label{spheres} Let $(\X, \|\  \|)$ be a Banach space.
Suppose that $0<a\leq\|x\|\leq \|x'\|\leq b<1$ for some $a, b\in \R$ and $x, x'\in \X$.
Then
 $$\|x-x'\|\leq b\Bigg\|{x\over{\|x\|}}-{x'\over{\|x'\|}}\Bigg\|+(b-a).$$
 \end{lemma}

\begin{lemma}\label{squizing} Suppose that $\X, \Y$ are Banach spaces and $T:\X\rightarrow \Y$ is an injective operator.
Then for $n\in \N$ there are subsets $\X_n$ of the  unit sphere $S_{\X, \|\  \|_T}$  and constants $0<c_n<1$ such that 
 $S_{\X, \|\ \|_T}=\bigcup_{n\in \N}\X_n$   and for any $x, x'\in \X_n$ for any $n\in \N$
 $$\|x-x'\|_\X\leq(1-c_n)\|(x/\|x\|_\X)-(y/\|y\|_\X)\|_\X+c_n/4.$$
\end{lemma}
\begin{proof}For $k\in\N\setminus\{0\}$ and $0\leq i\leq 4k-2$ let 
$$\X_{k, i}=\{x\in S_{\X, \|\  \|_T}: { i\over4k}\leq \|x\|_\X\leq {i+1\over4k}\ \&\ \|T(x)\|_\Y>1/k\}.$$
 Since $T$ is injective
$S_{\X, \|\ \|_T}=\bigcup\{\X_{k, i}: k\in\N\setminus\{0\},\  0\leq i\leq 4k-2\}$.
Applying Lemma \ref{spheres} to $x, x'\in \X_{k, i}$ we obtain
$$\|x-x'\|_\X\leq(1-1/k)\|(x/\|x\|_\X)-(x'/\|x'\|_\X)\|_\X+1/4k,$$
so by renumerating $\X_{k, i}$s as $\X_n$s we obtain the lemma.
\end{proof} 

\begin{lemma}\label{symptom-renorm} Suppose that $(\X, \|\ \|_\X)$ and $(\Y, \|\ \|_\Y)$
are Banach spaces and suppose that  $T:\X\rightarrow \Y$ is a linear bounded operator and $x, x'\in S_\X$
satisfy $\|x-x'\|_\X= 2-\varepsilon'$
 for some $\varepsilon'\geq0$ and $\|T(x)\|_\Y, \|T(x')\|_\Y\leq 1-\delta$ for some $0<\delta<1$.
Then 
$$\|x/\|x\|_T-x'/\|x'\|_T\|_T\geq 1+\varepsilon-\varepsilon',$$
where $\varepsilon={{2}\over{2-\delta}}-1$.
In particular, if $\delta=2/3$, $\varepsilon'=0$ and $\kappa$ is a cardinal, then the unit sphere of $(\X, \|\ \|_T)$ admits
a $(1+{1\over 2})$-separated set of cardinality $\kappa$ if the unit sphere
sphere of $(\X, \|\ \|)$ admits
a $2$-separated set of cardinality $\kappa$ on which $T$ is bounded by $1/3$.
\end{lemma}

\begin{proof} 
Consider
$y=x/\|x\|_T$ and $y'=x'/\|x'\|_T$.
As $$\|y-y'\|_T=\|y-y'\|_{\X}+\|T(y)-T(y')\|_\Y,$$
it is enough to prove that
$\|y-y'\|_{\X}\geq 1+\varepsilon-\varepsilon'$. We have
$$\|y-y'\|_{\X}=
\|{{x}\over{\|x\|_T}}-{{x'}\over{\|x'\|_T}}\|_\X
\geq \|x-x'\|_\X-\|{{x}\over{\|x\|_T}}-x\|_\X
-\|{{x'}\over{\|x'\|_T}}-x'\|_\X.$$

Since  we have $\|x\|_\X=\|x'\|_\X=1$,
 we estimate 
$$\|{{x}\over{\|x\|_T}}-x\|_\X=
\Big| {1\over{\|x\|_\X+\|T(x)\|_\Y}}
 -1\Big|\|x\|_\X=\Big| {1\over{1+\|T(x)\|_\Y}} -1\Big|.
$$
Since we have $\|T(x)\|_\Y\leq 1-\delta$, we obtain that
$${{1+\varepsilon}\over{2}}= {1\over{2-\delta}}\leq {1\over{1+\|T(x)\|_\Y}} \leq 1,$$
and so
$$\|{{x}\over{\|x\|_T}}-x\|_\X\leq 1-{{1+\varepsilon}\over{2}}={{1-\varepsilon}\over{2}}.$$
The same calculation works for $\|{{x'}\over{\|x'\|_T}}-x'\|_\X,$
so using $\|x-x'\|_\X=2-\varepsilon'$
we conclude that 
$$\|y-y'\|_{T}\geq \|y-y'\|_{\X}\geq 2-\varepsilon'-2({{1-\varepsilon}\over{2}})= 1+\varepsilon-\varepsilon'$$ as required.

\end{proof}

If the range of $T:\X\rightarrow \Y$ is separable, then infinitary combinatorics can play
a big role in the renorming $\|\ \|_T$ of $\X$ due to the following trivial fact:

\begin{lemma}\label{separable} Suppose that $\X, \Y$ are Banach spaces and $T:\X\rightarrow \Y$ has separable range.
Then for every $\varepsilon>0$  there are subsets $\Y_j$ for $j\in \N$ of the  unit sphere $S_{\X, \|\  \|_T}$ such that 
  for any $x, x'\in \Y_j$ for any $j\in \N$
 $$\|T(x)-T(x')\|_\Y\leq\varepsilon.$$
\end{lemma}
\begin{proof} For metric spaces the separability is equivalent to the Lindel\"of property, so we can
cover the range of $T$ by countably many balls of diameter $\varepsilon$. The $Y_n$s are preimages of 
these balls under $T$ intersected with $S_{\X, \|\  \|_T}$.
\end{proof}

 \begin{proposition}\label{sigma-T} Suppose that $\A$ is an $\ad$-family such that
 $\PP_\A$ is $\sigma$-centered and $T:\X_\A\rightarrow \Y$ 
 is a bounded linear and injective operator with separable range into a  Banach space $\Y$.  
 Then the unit sphere of $(\X_\A, \|\  \|_T)$ is
 the union of countably many open sets of diameters strictly less than $1$.
 \end{proposition}
 \begin{proof}
 For $m, k\in \N$, $k>0$ let $\Z_{m, k}\subseteq S_{\X_\A, \|\ \|_\infty}$ be
  of $\|\ \|_{\infty}$-diameter not bigger that $1+1/k$ and such that
 $S_{\X_\A, \|\ \|_\infty}=\bigcup_{m\in \N}\Z_{m, k}$ for every $k\in \N\setminus\{0\}$. The existence of these sets follows from
 Proposition \ref{sphere-precaliber}. Let $\X_n$ and $c_n$ be as in Lemma \ref{squizing}. By Lemma \ref{squizing}
 it is enough to partition each $\X_n$ into countably many sets of diameters strictly less than $1$. So fix $n\in \N$.
 
 As $0<c_n<1$ we have $(4-3c_n)/(4-4c_n)>1$ and so there is $k\in \N\setminus\{0\}$  such that  $1+1/k$ is smaller than this value
 and so 
 $$(1-c_n)(1+1/k)<{(1-c_n)(4-3c_n)\over(4-4c_n)}=1-{3c_n\over 4}.$$
 If $x/\|x\|_{\infty}, x'/\|x'\|_{\infty}\in \Z_{m, k}$ and $x, x'\in \Y_j$ for $j, m\in \N$, 
 where $\Y_j$ is as in Lemma \ref{separable} 
 for $\varepsilon=c_n/4$, Lemma \ref{squizing}
 implies that 
 $$\|x-x'\|_T=\|x-x'\|_{\infty}+\|T(x)-T(x')\|_\Y\leq 1-{3c_n\over 4}+c_n/4+c_n/4=1-c_n/4.$$
 As $\X_n=\bigcup_ {j, m\in \N} (\{x\in S_{\X_\A, \|\ \|_T}: x/\|x\|\in \Z_{m, k}\}\cap \Y_j)$,  
 we have obtained the required partition  of $\X_n$
 which completes the proof of the proposition.
 
 \end{proof}
 
 \begin{proposition}\label{precaliber-T} Suppose that $\kappa$ is a cardinal  of uncountable
  cofinality and $\A$ is an $\ad$-family such that
 $\PP_\A$ has precaliber $\kappa$ and $T:\X_\A\rightarrow \Y$ 
 is a bounded linear and injective operator with separable range into a  Banach space $\Y$.  
 Then every set  of unit vectors  in $(\X_\A, \|\ \|_T)$  of cardinality $\kappa$
 contains a subset of cardinality $\kappa$ which has the diameter 
 not bigger than  $1-\varepsilon$ for some $\varepsilon>0$.
 \end{proposition}
 \begin{proof} Let $\mathcal U\subseteq S_{\X_\A, \|\ \|_T}$ be of cardinality $\kappa$. 
 Let $\X_n$ and $c_n$ be as in Lemma \ref{squizing}.   Since  the $\X_n$s cover $S_{\X_\A, \|\ \|_T}$, by the
 uncountable cofinality of $\kappa$ there is $n\in \N$ such that $\mathcal U\cap \X_n$ 
 has cardinality $\kappa$.
 Fix such an $n\in \N$ and let us prove that $\mathcal U\cap \X_n$ 
 contains a subset of cardinality $\kappa$ which has the diameter not bigger than $1-\varepsilon$ for some $\varepsilon>0$ which
 is sufficient to prove the proposition.
 
 As $0<c_n<1$ we have $(4-3c_n)/(4-4c_n)>1$ and so there is $k\in \N\setminus\{0\}$  such that  $1+1/k$ is smaller than this value.
 In particular we have
  $$(1-c_n)(1+1/k)<{(1-c_n)(4-3c_n)\over(4-4c_n)}=1-{3c_n\over 4}.$$
 By passing to a subset of cardinality $\kappa$, by 
 Proposition \ref{sphere-precaliber} we may assume that  
 $$\{x/\|x\|_{\infty}:  x\in \mathcal U\cap \X_n\}$$
  has $\|\ \|_{\infty}$-diameter not bigger than $(1+1/k)$,
 and so Lemma \ref{squizing} implies that 
 $$\|x-x'\|_{\infty}\leq 1-{3c_n\over 4}+c_n/4=1-c_n/2$$
 for any $x, x'\in \mathcal U\cap \X_n$.  Using 
  Lemma \ref{separable} for $\varepsilon=c_n/4$  and the uncountable cofinality of $\kappa$ 
  we may find $\mathcal V\subseteq \mathcal U\cap \X_n$ of cardinality $\kappa$ such that $\|T(x)-T(x')\|_\Y\leq c_n/4$ 
  for every $x, x'\in \mathcal V$. So for $x, x'\in \mathcal V$ we have
 $$\|x-x'\|_T=\|x-x'\|_{\infty}+\|T(x)-T(x')\|_\Y\leq 1-{c_n\over 2}+c_n/4=1-c_n/4,$$
 which completes the proof of the proposition.
 
 \end{proof}
 
  \begin{proposition}\label{ccc-T} Suppose  that $\A$ is an $\ad$-family such that
 $\PP_\A$ satisfies the c.c.c.  and $T:\X_\A\rightarrow \Y$ 
 is a bounded linear and injective operator with separable range into a  Banach space $\Y$.  Then the unit sphere of  $(\X_\A, \|\  \|_T)$  
 does not admit an uncountable $1$-separated set.
 \end{proposition}
 \begin{proof} Let $\mathcal U\subseteq S_{\X_\A, \|\  \|_T}$ be of cardinality $\omega_1$. We will
 find two elements of $\mathcal U$ which are $\| \ \|_{T}$-distant  by less than $1$.
 Let $\X_n$ and $c_n$ be as in Lemma \ref{squizing}.   Since  the $\X_n$s cover $S_{\X_\A, \|\ \|_T}$ 
 there is $n\in \N$ such that $\mathcal U\cap \X_n$ is uncountable.
 Fix such an $n\in \N$ and let us prove that $\mathcal U\cap \X_n$ 
 contains two elements  $\| \ \|_{T}$-distant  by less than $1$.
 
  Using 
  Lemma \ref{separable} for $\varepsilon=c_n/4$  
  we may find an uncountable $\mathcal V\subseteq U\cap \X_n$ such that $\|T(x)-T(x')\|_\Y\leq c_n/4$ 
  for every $x, x'\in \mathcal V$.
 As $0<c_n<1$ we have $(4-3c_n)/(4-4c_n)>1$ and so there is $k\in \N\setminus\{0\}$  such that  $1+1/k$ is smaller than this value.
 In particular we have
  $$(1-c_n)(1+1/k)<{(1-c_n)(4-3c_n)\over(4-4c_n)}=1-{3c_n\over 4}.$$
Lemma \ref{squizing} implies that 
 $$\|x-x'\|_{\infty}\leq (1-c_n)\|x/\|x\|_{\infty}-x'/\|x'\|_{\infty}\|+c_n/4$$
 for any $x, x'\in \mathcal V\cap \X_n$.  
 Now apply Proposition \ref{sphere-precaliber} for $\{x/\|x\|_{\infty}: x\in V\cap \X_n\}$
 finding $x, x'\in V\cap \X_n$ such that $\|x/\|x\|_{\infty}-x'/\|x'\|_{\infty}\|<1+1/k$.
  So for these $x, x'\in \mathcal V$ we have
 $$\|x-x'\|_T=\|x-x'\|_{\infty}+\|T(x)-T(x')\|_\Y\leq 1-{3c_n\over 4}+c_n/4+c_n/4=1-c_n/4,$$
 which completes the proof of the proposition.
 \end{proof}
 
 \begin{proposition}\label{sphereT-PA} Suppose that $\A$ is an $\ad$-family,  Then
the following implications hold.
\begin{enumerate}
\item If $S_{\X_\A, \|\ \|_\T}$ is the union of countably many sets of diameters not bigger than $1$,
then $\PP_\A$ is $\sigma$-centered.
\item If every subset of $S_{\X_\A, \|\ \|_\T}$ of cardinality $\kappa$ of
uncountable cofinality contains a further subset which has the diameter not bigger than $1$, then $\PP_\A$ has precaliber $\kappa$.
\item  If $S_{\X_\A, \|\ \|_\T}$ does not admits an uncountable $(2-\varepsilon)$-separated set
for some $\varepsilon>0$, then
 $\PP_A$ satisfies the c.c.c.

\end{enumerate}
\end{proposition} 
 \begin{proof}
 
In the parts (1), (2) and (3)
it is enough to prove that there is $m\in \N$ such that
$$\Q_m=\{p\in \PP_\A: \{0,\dots, m\}\cap(A_p\cup B_p)=\emptyset\}$$
is $\sigma$-centered, has precaliber $\kappa$ and satisfies the c.c.c. respectively.
This is because for a fixed $m\in \N$ the partial order  $\PP_\A$ consists of finitely many parts each isomorphic to
$\Q_m$ which depend on the intersections of
$A_p$ and $B_p$ with $\{0,\dots, m\}$ for $p\in \PP_\A$.
 
 In the proofs of all parts above we will use the fact that $\|T(x)\|\leq 1/m$
 if $x|\{0, \dots, m\}=0$, if $T$ is as in the definition of the norm $\|\ \|_\T$ and $\|x\|_\infty\leq 1$. In particular,
 if $p,q\in \Q_m$, then Lemma \ref{symptom-renorm} can 
 be applied for $\delta=1-1/m$ to the elements $f_p, f_q$ defined in Definition \ref{fp}, 
 i.e., by Lemma \ref{cond-and-function} (1), (2)
 we have 
 $$\|f_p/\|f_p\|_\T-f_q/\|f_q\|_\T\|_\T\geq {2\over 2-\delta}={2\over 1+1/m}=2-{2\over m+1}\leqno(*)$$
  if  $p,q\in \Q_m$ are incompatible.
 
 Also note that a finite set $\PP\subseteq\PP_\A$ is centered if and only if
 $\bigcup\{A_p: p\in \PP\}\cap \bigcup\{B_p: p\in \PP\}=\emptyset$ if and
 only if $\PP$ is pairwise compatible.
 
(1) Let $\X_n$'s form a cover of  $S_{\X_\A, \|\ \|_\T}$ by sets of $\|\ \|_\T$-diametres not bigger  than $1$. 
For $n\in \N$ consider
$$\PP_n=\{p\in \Q_3: f_p/\|f_p\|_\T\in \mathcal X_n\}.$$
Then $\bigcup_{n\in \N}\PP_n=\Q_3$ and by (*) each $\PP_n$
is pairwise compatible and so it is centered.

(2) Given a set $\{p_\xi: \xi<\kappa\}\subseteq\Q_3$, consider
$\{f_{p_\xi}/\|f_{p_\xi}\|_\T: \xi<\kappa\}\subseteq S_{\X_\A, \|\ \|_\T}$. By the hypothesis
there is $\Gamma\subseteq\kappa$ of cardinality $\kappa$ such that
$\{f_{p_\xi}/\|f_{p_\xi}\|_\T: \xi\in \Gamma\}$ has the diameter not bigger than $1$.
As in the proof of (1)  by (*) we obtain that
$\{p_\xi: \xi\in \Gamma\}$ is pairwise compatible, and so as required.

(3) Fix $\varepsilon>0$. Let $m\in \N$ be
such that $2/(1+m)<\varepsilon$. Given a set $\{p_\xi: \xi<\omega_1\}\subseteq\Q_m$, consider
$\{f_{p_\xi}/\|f_{p_\xi}\|_\T: \xi<\kappa\}\subseteq S_{\X_\A, \|\ \|_\T}$. By the hypothesis
there are $\xi<\eta<\omega_1$ such that $\|f_{p_\xi}/\|f_{p_\xi}\|_\T-f_{p_\eta}/\|f_{p_\eta}\|_\T\|\leq 2-\varepsilon
<2-2/(1+m)$.
 As in the proof of (1) by (*) $p_\xi$ and $p_\eta$ are compatible, as required.
\end{proof}

\begin{proposition}\label{oca-T} Assume {\sf OCA}. 
Whenever $\A$ is an $\ad$-family, then either
\begin{enumerate}
\item The unit sphere of $(\X_\A, \|\ \|_\T)$ is the union of countably many sets
of diameters less than $1$ or
\item for each $0<\varepsilon<1$ there is an uncountable   $(1+\varepsilon)$-separated subset of the unit sphere of 
$(\X_\A, \|\ \|_\T)$.
\end{enumerate}
\end{proposition}
\begin{proof}
As proved in Proposition \ref{oca}, the Open Coloring Axiom implies that
for every $\ad$-family $\A$ the forcing $\PP_\A$ is either $\sigma$-centered or 
fails to satisfy the c.c.c.

If $\PP_\A$ admits an uncountable antichain, then $S_{\X_\A, \|\ \|_\T}$
admits a $(1+\varepsilon)$-separated set for every  $0<\varepsilon<1$ by Proposition \ref{sphereT-PA}  (3).

If $\PP_\A$ is $\sigma$-centered, then $S_{\X_\A, \|\ \|_\T}$ 
is the union of countably many sets of diameter less than $1$ by Proposition \ref{sigma-T}.
\end{proof}

The following proposition shows that the above dichotomy cannot be proved in {\sf ZFC} alone.

\begin{proposition}\label{antiramsey-T}
Let $\A$ be an antiramsey $\ad$-family. Then for every injective bounded linear operator
$T: \X_\A\rightarrow \Y$ with separable range and every Banach space $\Y$ and every $\varepsilon>0$
there is a set $\{x_\xi: \xi<\omega_1\}\subseteq S_{\X_\A, \|\ \|_T}$
such that for every uncountable $\Gamma\subseteq\omega_1$
\begin{itemize} 
\item there are $\xi,\eta\in \Gamma$ such that $0<\|x_\xi-x_\eta\|_T<1$ and
\item there are $\xi,\eta\in \Gamma$ such that $\|x_\xi-x_\eta\|_T\geq2-\varepsilon$ and
\end{itemize}
\end{proposition}
\begin{proof}
Let $\A=\{A_\alpha: \alpha<\omega_1\}$. Let $k\in \N$ be such that $2/(k+1)\leq\varepsilon$.

 By Lemma \ref{separable} there is an uncountable $\Theta\subseteq\omega_1$
such that $\|T(1_{A_\alpha}-1_{A_\beta})\|_\Y\leq 1/k$ for every $\alpha, \beta\in \Theta$.
Let $\{\alpha_\xi: \xi<\omega_1\}$ be the increasing enumeration of $\Theta$ and 
let
$x_\xi'=1_{A_{\alpha_\xi}}-1_{A_{\alpha_{\xi+1}}}$ and $x_\xi=x_\xi'/\|x_\xi'\|_T$.  Let $\Gamma\subseteq\omega_1$ 
be uncountable.

Since $\PP_\A$ satisfies the c.c.c. when $\A$ is antiramsey (Definition \ref{def-antiramsey}, 
Proposition \ref{ccc-T} implies that there are $\xi,\eta\in \Gamma$ such that $0<\|x_\xi-x_\eta\|_T<1$.

Let $p_\xi=(A_{\alpha_\xi}\setminus A_{\alpha_{\xi+1}}, A_{\alpha_{\xi+1}}\setminus A_{\alpha_{\xi}})\in \PP_\A$.
Note that $x_\xi'=f_{p_\xi}$ (see Definition \ref{fp}).
Since $\A$ is antiramsey there are $\xi<\eta$ such that $p_\xi$ and $p_\eta$ are not compatible.
By Lemma \ref{cond-and-function} (2) we obtain that $\|x_\xi'-x_\eta'\|_\infty=2$.  Taking
$\delta=1-1/k$ in Lemma \ref{symptom-renorm} (which is justified by the choice of $\Theta$) 
we obtain that 
$$\|x_\xi-x_\eta\|_T\geq {2\over 2-\delta}={2\over 1+1/k}=2-{2\over k+1}\geq  2-\varepsilon,$$ as required.

\end{proof}

\section{Renorming of a subspace of  $\X_{\A,\phi}$ for an antiramsey $\ad$-family $\A$}

In this section we consider a space of the form $\X_{\A, \phi}$
for an $\ad$-family $\A$ and a pairing $\phi$ (see Definitions \ref{pairing} and \ref{def-subspace}))
with the norm $\|\ \|_\T$ defined at the beginning of the previous section.

\begin{proposition}\label{antiramsey-Tphi} Suppose that $\kappa$ is an uncountable
cardinal, $\A$ is an antiramsey $\ad$-family
of cardinality $\kappa$ and $\rho>0$. There is a pairing
$\phi:\kappa\rightarrow \A$ and a subspace $\Y$ (both depending on $\rho$) of $\X_{\A, \phi}$ such that 
whenever $\{y_\xi: \xi<\omega_1\}\subseteq S_{\Y, \|\ \|_\T}$ for each $\xi<\eta<\omega_1$ satisfies
$$\|y_\xi-y_\eta\|_\T\geq 1-\varepsilon,$$
for some $\varepsilon\leq1/9$, then 
\begin{enumerate}
\item there are $\xi<\eta$ such that $\|y_\xi-y_\eta\|_\T>2-24\varepsilon-\rho$ and
\item there are $\xi<\eta$ such that $\|y_\xi- y_\eta\|_\T<1$.
\end{enumerate}
\end{proposition}
\begin{proof} 
Let $T$ be as in the definition of $\|\ \|_\T$. Let $m\in \N$ be such that $2/(m+1)<\rho$.
Then $T$ restricted to 
$$\X_m=\{x\in \X_\A: x|\{0, \dots, m\}=0\}$$ has norm not bigger than $1/m$. 
Let $\phi:\kappa\rightarrow \A$ be a pairing such that $\phi_1(\alpha)\cap \{0,\dots, m\}=
\phi_{-1}(\alpha)\cap \{0,\dots, m\}$ for every $\alpha<\kappa$. 
Then the space $\Y$ generated by both
$$\{x\in c_0: x|\{0, \dots, m\}=0\}\  \hbox{\rm and}\ 
\{1_{\phi_1(\alpha)}-1_{\phi_{-1}(\alpha)}: \alpha<\kappa\}$$  is included in
$\X_m\cap \X_{\A, \phi}$, so again  $T$ restricted to it  has norm not bigger than $1/m$
and so Lemma \ref{symptom-renorm} can be applied to $\Y$ and $\delta=1-1/m$.

By passing to an uncountable set and using Lemma \ref{separable} we may assume that 
two conditions are valid, namely first that
$\|T(y_\xi)-T(y_\eta)\|_2\leq\varepsilon$ 
and second that $|\|y_\xi\|_\infty-\|y_\eta\|_\infty|\leq \varepsilon$ for every $\xi, \eta\in \omega_1$.
The first condition and the hypothesis on $\{y_\xi: \xi<\omega\}$
imply that $\|y_\xi-y_\eta\|_\infty\geq 1-2\varepsilon$ for every $\xi<\eta<\omega_1$, as 
$\|y_\xi\|_\infty<
\|y_\xi\|_\T=1$ and  $\|y_\eta\|_\infty<
\|y_\eta\|_\T=1$, by the
injectivity of $T$. The second condition allows us to apply Lemma \ref{spheres}
for $b=\max(\|y_\xi\|_\infty, \|y_\eta\|_\infty)$, $a=\max(\|y_\xi\|_\infty, \|y_\eta\|_\infty)$
and $(b-a)\leq\varepsilon$
to conclude that  
$$b\|y_\xi/\|y_\xi\|_\infty-y_\eta/\|y_\eta\|_\infty\|_\infty+\varepsilon\geq
\|y_\xi-y_\eta\|_\infty\geq 1-2\varepsilon$$
and so
$$\|y_\xi/\|y_\xi\|_\infty-y_\eta/\|y_\eta\|_\infty\|_\infty>1-3\varepsilon$$
for any $\xi<\eta<\omega_1$.
Moreover $3\varepsilon\leq 1/3$, so
by Proposition \ref{Y-antiramsey} there are $\xi<\eta<\omega_1$ such that
 $\|y_\xi/\|y_\xi\|_\infty-y_\eta/\|y_\eta\|_\infty\|_\infty>2-24\varepsilon$.
 Let $\varepsilon'\geq 0$ be such that
 $\|y_\xi/\|y_\xi\|_\infty-y_\eta/\|y_\eta\|_\infty\|_\infty=2-\varepsilon'$. Clearly
 $-\varepsilon'>-24\varepsilon$. 
 Before applying \ref{symptom-renorm} note that
$(y_\xi/\|y_\xi\|_\infty)/\|y_\xi/\|y_\xi\|_\infty\|_\T=y_\xi/\|y_\xi\|_\T=y_\xi$
and similarly for $y_\eta$,
 So Lemma \ref{symptom-renorm}
yields
$$\|y_\xi-y_\eta\|_\T\geq {2\over 2-\delta}-\varepsilon'=
{2\over 1+1/m}-\varepsilon'=2-\varepsilon'-{2\over m+1}>
2-24\varepsilon-\rho$$ 
as required in the case of (1).

By the definition of an  antiramsey $\ad$-family, the partial order
of $\PP_\A$ satsifies the c.c.c. and so by Proposition \ref{ccc-T} there are
$\xi<\eta<\omega_1$ like in (2).
\end{proof}

\section{Questions}

As the properties of the $\ad$-families considered in this paper and the geometric phenomena on
the induced Banach spaces are completely new, 
the results of this paper generate many questions which the authors have not answered. 
Below we mention several of such questions
which seem most interesting and natural at this moment to the authors.

\begin{question}
Can we consistently improve the dichotomy \ref{oca} obtaining: for any $\ad$-family  $\A$
the partial order $\PP_\A$ is either $\sigma$-centered or it contains an antichain of cardinality $|\A|$?
\end{question}

This would be  equivalent to obtaining a dichotomy for Banach spaces $(\X_\A, \|\ \|_T)$:
Banach spaces $(\X_\A, \|\ \|_T)$ either contain a $(1+\varepsilon)$-separated set
of cardinality equal to the density of the Banach space or the unit sphere of it is 
the union of countably many sets which have diameter strictly less than $1$.

The defintion of a Luzin family (Definition \ref{def-L}) makes sense only for
$\ad$-families of cardinality $\omega_1$. However the strong consequences of being
Luzin like being an $L$-family in principle could hold for $\ad$-families of bigger cardinalities:

\begin{question}\label{questionL} Is it consistent that  there are $L$-families of cardinalities bigger than $\omega_1$?
\end{question}

In fact one could introduce a notion symmetric to $\PP_\A$ having precaliber $\omega_1$:

\begin{definition} An $\ad$-family is called $L$-saturated 
if every uncountable subset of $\PP_\A$ consisting of pairwise essentially distinct elements
contains and uncountable antichain.
\end{definition}

\begin{question}\label{questionLsat} Is it consistent that  there 
are $L$-saturated $\ad$-families of cardinalities bigger than $\omega_1$?
\end{question}

Note that in Questions \ref{questionL} and \ref{questionLsat} we can only hope for
the consistency because it was proved in \cite{oursacks} that 
it is consistent that every $\ad$-family of cardinality $\mathfrak c=\omega_2$
contains a subfamily of cardinality $\omega_2$ which is $\R$-embeddable
which yields a big pairwise compatible set in $\PP_\A$ by Proposition \ref{Rembed-centered}.
We also do not know the answers to the following questions:

\begin{question} Is it consistent that there is an $\ad$-family such that
$\PP_\A$ has precaliber $\omega_1$ but $\PP_\A$ is not $\sigma$-centered?
\end{question}

\begin{question} Is it consistent that there is an $\ad$-family  of cardinality $\omega_1$
which is $L$-saturated but is not an $L$-family?
\end{question}

\begin{question} Is there (consistently) an $L$-family which is not a Luzin family?
\end{question}

It is clear in the light of the results of the last three sections of this paper  that any answer to the above questions would
generate a result concerning Banach spaces.

The dichotomy \ref{oca} holds for Banach spaces of the form $(\X_\A, \|\ \|_\T)$ under 
an additional axiom {\sf OCA}.
It is very natural to ask if this can be extended for bigger classes of Banach spaces,
for example for all Banach spaces:

\begin{question} Is it consistent that for  every nonseparable Banach space the unit sphere
either admits an uncountable $(1+)$-separated set
or is the union of countably many sets of diameters strictly less than $1$?
\end{question}

\begin{question} Is there in {\sf ZFC} a  nonseparable Banach space $\X$ and $\delta>0$
such that for sufficiently small  $\varepsilon>0$  every
$(1-\varepsilon)$-separated set in $S_\X$ contains two elements distant by less than $1$
and two elements distant by more than $1$?
\end{question}

\begin{question} Is there in {\sf ZFC} a  nonseparable Banach space $\X$ 
such that for sufficiently small  $\varepsilon>0$  every
$(1-\varepsilon)$-separated set in $S_\X$ contains two elements distant by less than $1+\delta$
and two elements distant by more than $1+2\delta$?
\end{question}

Let us also repeat a relevant question from \cite{n-luzin} (cf.  Proposition  \ref{oca-L})

\begin{question} Is it consistent with {\sf MA}$+\neg${\sf CH} that there is
an $\ad$-family admitting a $3$-Luzin gap but not admitting a $2$-Luzin gap?
\end{question}

\bibliographystyle{amsplain}

\end{document}